\pdfoutput=1

\documentclass[12pt,a4paper]{article}

\usepackage{amssymb, amsmath, amsthm}

\usepackage{geometry}
\usepackage{authblk}

\usepackage{cite}

\usepackage{graphicx}

\usepackage{booktabs}

\newtheorem{thm}{Theorem}

\usepackage{caption}
\usepackage{pdfpages}

\title{Collocation-based harmonic balance framework for highly accurate periodic solution of nonlinear dynamical system}

\author[,$\dagger$,a,b]{Honghua Dai\thanks{Corresponding author: E-mail: hhdai@nwpu.edu.cn}}
\author[,a,b]{Zipu Yan\thanks{These authors contributed equally to this work.}}
\author[a,b,c]{Xuechuan Wang}
\author[a,b]{Xiaokui Yue}
\author[c]{Satya N. Atluri}
\affil[a]{School of Astronautics, Northwestern Polytechnical University, 127 West Youyi Road, Xi'an, 710072, Shaanxi, China}
\affil[b]{National Key Laboratory of Aerospace Flight Dynamics, Northwestern Polytechnical University, 127 West Youyi Road, Xi'an, 710072, Shaanxi, China}
\affil[c]{Department of Mechanical Engineering, Texas Tech University, 2500 Broadway, Lubbock, 79401, Texas, USA}

\begin{document}
\maketitle


\begin{abstract}
Periodic dynamical systems ubiquitously exist in science and engineering. The harmonic balance (HB) method and its variants have been the most widely-used approaches for such systems, but are either confined to low-order approximations or impaired by aliasing and improper-sampling problems. Here we propose a collocation-based harmonic balance framework to successfully unify and reconstruct the HB-like methods. Under this framework a new conditional identity, which exactly bridges the gap between frequency-domain and time-domain harmonic analyses, is discovered by introducing a novel aliasing matrix. Upon enforcing the aliasing matrix to vanish, we propose a powerful reconstruction harmonic balance (RHB) method that obtains extremely high-order ($>$100) non-aliasing solutions, previously deemed out-of-reach, for a range of complex nonlinear systems including the cavitation bubble equation and the three-body problem. We show that the present method is 2-3 orders of magnitude more accurate and simultaneously much faster than the state-of-the-art method. Hence, it has immediate applications in multi-disciplinary problems where highly accurate periodic solutions are sought.\\

\end{abstract}

\section{Introduction}
Periodic response plays a pivotal role across many scientific and engineering disciplines, ranging from mechanical vibrations\cite{kapania1987buckling, mchugh2021flutter, li2011ambient, ramananarivo2011rather}, fluid dynamics\cite{marmottant2003controlled, ambrose2010computation}, to celestial dynamics\cite{stone2019statistical, howard2000stability, yu2012generating}. Solving periodic solution to linear system is trivial with a batch of ready-to-use methods\cite{schiff1999laplace, blaquiere2012nonlinear}. For nonlinear problems, however, analytical methods become much more complicated due to the nullification of the additivity and homogeneity properties\cite{hernandez2019novel, baoyin2006solar}. Numerical integration methods can be used to find periodic nonlinear solutions\cite{liu2001cone, junkins2013picard, dai2015comparison, cordoba2005evidence}, but limited by i) undesired simulation time for transient motions; ii) small step-size to constrain accumulated computational error; and iii) inability to compute unstable periodic solutions. Credited to the seminal work of Blondel\cite{blondel1919amplitude}, the HB method has been the most popular method for nonlinear periodic solutions in various fields during the last 100 years\cite{kryloff1947introduction, stoker1950nonlinear, mickens1984comments, hall2002computation, saetti2021revisited}. It is free from the above restrictions, and works by presuming a Fourier expansion for the desired periodic solution and then obtaining resultant nonlinear algebraic equations (NAEs) of the Fourier coefficients through balancing harmonics up to the truncation order. However, the derivation of HB algebraic system could be extremely lengthy and tedious\cite{hall2002computation, dai2012simple} even with the aid of algebra software\cite{dai2014time}. This computational barrier, which also exists in its close relatives, e.g., incremental HB method\cite{lau1981amplitude}, impedes the HB method from analyzing sophisticated problems where high order harmonics are indispensable.

An ingenious attempt to overcome the computational difficulty of HB is the alternating frequency-time harmonic balance (AFT) method\cite{cameron1989alternating}. It bypasses the complex symbolic calculations via computationally-cheap numerical calculations, thus finding wide applications\cite{kim1990bifurcation, hou2017application, van2019frequency}. Based on the Shannon sampling theorem\cite{shannon1949communication}, the sampling rate of AFT is required to exceed twice the highest signal frequency to recover the nonlinear responses. Unexpectedly, numerical experiments reveal that the sampling rate could be far less than the prediction of Shannon theorem\cite{krack2019harmonic}. Regarding that, a more concise and elegant high dimensional harmonic balance (HDHB) method was proposed\cite{hall2002computation, liu2006comparison} using collocation, where the collocation number equals to the number of unknown Fourier coefficients, which is much less than the sampling rate of the AFT method. The HDHB method is computationally fast\cite{thomas2002nonlinear, ekici2011harmonic}, and a series of modified versions, e.g., the Chebyshev-based Time-spectral Method (C-TSM)\cite{ekici2020modeling}, the Supplemental-frequency Harmonic Balance (SF-HB) method\cite{li2021supplemental}, etc., have been developed for specific problems. However, both the AFT method and the family of HDHB methods are impaired by aliasing when dealing with non-polynomial nonlinear problems\cite{mickens1984comments, krack2019harmonic}. Worse yet, severe aliasing may cause non-physical solutions in the HDHB-like methods even for polynomial nonlinearity\cite{liu2006comparison, dai2012simple, dai2014time}. Our previous study proved that the HDHB method is inherently not a variation of the HB method, but equivalently a time domain collocation method in disguise\cite{dai2012simple}. Based on this finding, we extended the time domain collocation method by collocating at more nodes, and numerically suppressed the generation of non-physical solutions by least square method\cite{dai2013time}. However, the theoretical dealiasing rule is still unclear.

Here we propose a collocation-based harmonic balance framework to unify the HB, the AFT, and the HDHB methods. Under this framework, the theoretical dealiasing rule is revealed and plainly expressed by a novel aliasing matrix. It further leads to the discovery of an unprecedented conditional identity between the Fourier coefficients obtained from time domain analysis and frequency domain analysis. With these developments, we propose an advanced method for periodic analysis named as reconstruction harmonic balance (RHB) method. We transparently show that the RHB method can equivalently transform into the HB method, the AFT method, or the HDHB method by choosing corresponding collocation number. The present study completely addresses the computational difficulties of the HB method (computational burden due to symbolic calculations), the AFT method (aliasing in non-polynomial nonlinearity and redundant samplings) and the HDHB method (aliasing-induced  non-physical solutions). With the proposed method, very high order solutions can be obtained for strongly nonlinear dynamical systems with high accuracy and little computational effort.

The performance of the RHB method is evaluated by three nonlinear examples from structural vibrations, bubble dynamics, and orbital mechanics. First, the features of the RHB method are explored using the classical Duffing oscillator. Then the dynamical equation of cavitation bubble is investigated using the RHB method to demonstrate its capability of dealing with non-polynomial nonlinearities without causing aliasing. Moreover, the attractive circular restricted three-body problem (CRTBP) is studied, where the RHB method shows the superiority of being a highly efficient and robust alternative to the state-of-the-art method for finding periodic orbits. As highlighted by the results, the RHB method could be revolutionary to the field of orbital design for deep space exploration. The numerical results in this paper verified the effectiveness of the proposed RHB method in analyzing a wide range of problems across multiple disciplines.

\section{Collocation-based harmonic balance framework}
\subsection{Principles}
A schematic diagram of the collocation-based harmonic balance framework is presented in Fig. \ref{fig:Fig1}. To solve for the periodic solution of a nonlinear dynamical system, the HB method maps the solution to Fourier coefficients in frequency domain, but it is bogged down in lengthy symbolic calculations for obtaining the Fourier coefficients of nonlinear function. In contrast, the time domain analysis of the proposed framework is very straightforward and computationally economic by mapping the solution to collocation nodes. Moreover, it is strictly proven that the algebraic system of the collocation-based harmonic balance is equivalent to that of the HB method when the dealiasing condition is fulfilled, i.e. $M>(\phi+1)N$, where $M$ is the collocation number, $\phi$ is the degree of nonlinearity, and $N$ is the order of the method (number of harmonics retained). The HDHB method is proven to be equivalent to a time domain collocation method\cite{dai2012simple}, but its number of collocation nodes is chosen as $M=2N+1$. As indicated in Fig. \ref{fig:Fig1}, the collocation number of HDHB is insufficient to aviod aliasing, because high order harmonics between $N$ and $\phi N$ are misinterpreted as low order harmonics. However, by rising the number of collocation nodes, which is tantamount to increasing the sampling rate, the proposed collocation framework can prevent the unwanted high order harmonics from being mistakenly mixed into the low order harmonics. With enough collocation nodes, all the nonlinearity-induced higher order harmonics   can be excluded. Based on the Shannon sampling theorem, the AFT method adopts $M\geq2\phi N+1$ to avoid aliasing\cite{krack2019harmonic} (see Supplementary Part 3), but it is not the optimal sampling rate as observed in numerical experiments. According to the aliasing matrix that is first found in this paper, the number of collocations can be much lower than the AFT sampling rate. In Fig. \ref{fig:Fig1}, the non-zero elements of the aliasing matrix are marked as red, implying that the harmonics correlated by these elements are mixed. The non-zero elements of aliasing matrix gradually decrease with the rising of collocation nodes. This property of aliasing matrix leads to the discovery of a new conditional identity (see Supplementary Part 4) that links the nonlinear Fourier components in the time domain and the frequency domain. This conditional equivalence may find applications in a wide range of problems where Fourier analysis is employed. In the proposed RHB method, where $M=(\phi+1)N+1$, the aliasing matrix becomes exactly zero. Unless otherwise mentioned, the collocation number of the RHB method is set as $M=(\phi+1)N+1$ in the following computations.

\begin{figure}[tbhp]
	\centering
	\includegraphics[width=1\textwidth]{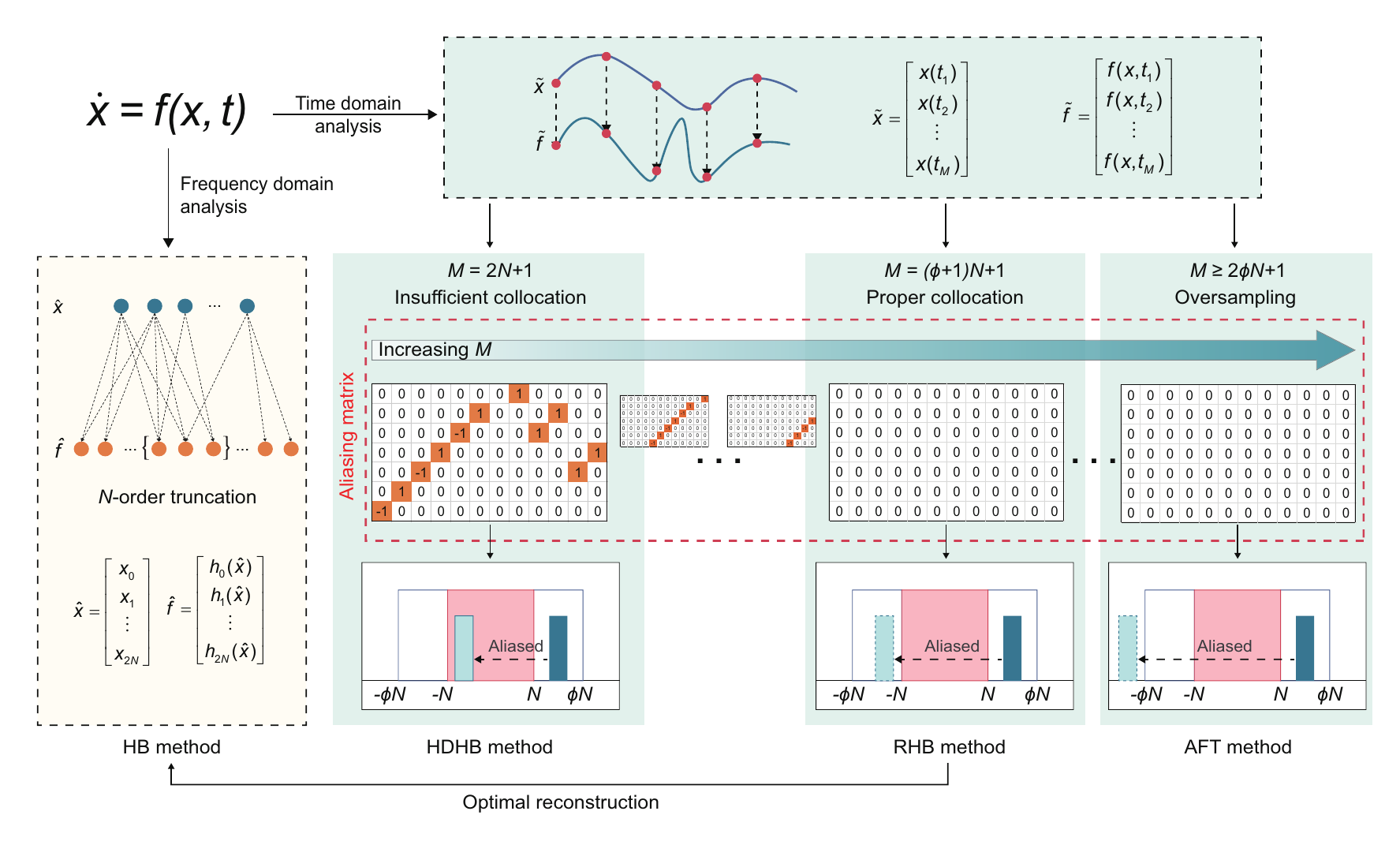}
	\caption{\textbf{Illustration of the collocation-based harmonic balance framework that unifies the HB, the HDHB, and the AFT methods.} Under the proposed framework, the RHB method is proposed to optimally reconstruct the HB method in time domain without causing aliasing. Unlike the classical harmonic balance method that seeks the coefficients of harmonics in frequency domain, the proposed framework solves for the collocation states in time domain, leading to three different methods, i.e., the HDHB, the RHB, and the AFT methods, by choosing different number of collocations.}
	\label{fig:Fig1}
\end{figure}

\subsection{Reconstruction harmonic balance method}
 Consider a general nonlinear dynamical system described by
\begin{equation}
	\dot{\mathbf{x}}=f(\mathbf{x},t).
	\label{con:DifferentialEquation}
\end{equation}
Since the periodic solution is our present interest, the state vector $\mathbf{x}$ can be approached by a truncated Fourier series in time. That is
\begin{equation}
	\mathbf{x}=\mathbf{x}_0+\sum_{n=1}^N[\mathbf{x}_{2n-1}\cos(n\omega t)+\mathbf{x}_{2n}\sin(n\omega t)],
	\label{con:xFourierForm}
\end{equation}
where $N$ is the truncation order, i.e., order of the method, and $\mathbf{x}_0$, $\mathbf{x}_1$, $\cdots$, $\mathbf{x}_{2n}$ are unknown Fourier coefficients. The first derivative of equation (\ref{con:xFourierForm}) with respect to time is
\begin{equation}
	\dot{\mathbf{x}}=\sum_{n=1}^N[-n\omega\mathbf{x}_{2n-1}\sin(n\omega t)+n\omega\mathbf{x}_{2n}\cos(n\omega t)].
	\label{con:xDotFourierForm}
\end{equation}
The classical HB method requires first substituting equations (\ref{con:xFourierForm})-(\ref{con:xDotFourierForm}) into the governing equation (\ref{con:DifferentialEquation}) and then balancing the Fourier coefficients of each harmonic up to the truncation order. Thus, equation (\ref{con:DifferentialEquation}) turns out to be the so-called HB algebraic system (see Supplementary Part 1)
\begin{equation}
	\omega\mathbf{A}\hat{\mathbf{x}}=\hat{f}(\hat{\mathbf{x}}),
	\label{con:HBForm}
\end{equation}
where
\[
\mathbf{A}=
\left[
\begin{array}{ccccc}
	0        &0               &0               &\cdots&0     \\
	0        &\mathbf{J}_{1}&0               &\cdots&0     \\
	0        &0               &\mathbf{J}_{2}&\cdots&0     \\
	\vdots&\vdots      &\vdots       &\ddots&\vdots\\
	0        &0              &0                &\cdots&\mathbf{J}_{N}
\end{array}
\right]
,\ \
\mathbf{J}_{n}=n
\left[
\begin{array}{cc}
	0&1\\-1&0
\end{array}
\right],
\]
and $\hat{f}$ is a nonlinear function of $\hat{\mathbf{x}}$; $\hat{\mathbf{x}}=[\mathbf{x}_0\ \ \mathbf{x}_1\ \ \cdots\ \ \mathbf{x}_{2n}]^\mathrm{T}$. Equation (\ref{con:HBForm}) contains $2N+1$ nonlinear algebraic equations (NAEs) to be solved by NAE solvers. To avoid symbolic calculations in the HB method, a series of time domain methods are developed\cite{hall2002computation,dai2012simple,krack2019harmonic}. Similar to the time domain methods, we establish the relation between the unknown Fourier coefficients and temporal quantities at $M$ equally spaced nodes over one period:
\begin{equation}
	\tilde{\mathbf{x}}=\mathbf{E}\hat{\mathbf{x}},
	\label{con:xtilde_equal_E_xhat}
\end{equation}
where $\tilde{\mathbf{x}}=[\mathbf{x}(t_1)\ \ \mathbf{x}(t_2)\ \ \cdots\ \ \mathbf{x}(t_M)]^\mathbf{T}$, and the collocation matrix $\mathbf{E}$ is
\[
\mathbf{E}=
\left[
\begin{array}{cccccc}
	1     &\cos(\omega t_1)  &\sin(\omega t_1)  &\cdots&\cos(N\omega t_1)  	&\sin(N\omega t_1)\\
	1     &\cos(\omega t_2)  &\sin(\omega t_2)  &\cdots&\cos(N\omega t_2)  	&\sin(N\omega t_2)\\
	\vdots&\vdots            &\vdots            &\ddots&\vdots             &\vdots           \\
	1     &\cos(\omega t_{M})&\sin(\omega t_{M})&\cdots&\cos(N\omega 	t_{M})&\sin(N\omega t_{M})
\end{array}
\right].
\]
Since rank($\mathbf{E}$)$=2N+1$, we have
\begin{equation}
	\mathbf{E^+E=I}_{2N+1},
	\label{con:EpE_equal_I}
\end{equation}
where $\mathbf{E^+}$ is the Moore-Penrose inverse matrix of $\mathbf{E}$ with the explicit expression being:
\[
\mathbf{E^+}=\dfrac{2}{M}
\left[
\begin{array}{cccc}
	\dfrac{1}{2}      &\dfrac{1}{2}      &\cdots&\dfrac{1}{2}        \\
	\cos(\omega t_1) &\cos(\omega t_2) &\cdots&\cos(\omega t_{M}) \\
	\sin(\omega t_1) &\sin(\omega t_2) &\cdots&\sin(\omega t_{M}) \\
	\cos(2\omega t_1)&\cos(2\omega t_2)&\cdots&\cos(2\omega t_{M})\\
	\sin(2\omega t_1)&\sin(2\omega t_2)&\cdots&\sin(2\omega t_{M})\\
	\vdots           &\vdots           &\ddots&\vdots             \\
	\cos(N\omega t_1)&\cos(N\omega t_2)&\cdots&\cos(N\omega t_{M})\\
	\sin(N\omega t_1)&\sin(N\omega t_2)&\cdots&\sin(N\omega t_{M})
\end{array}
\right].
\]
Multiplying both sides of equation (\ref{con:xtilde_equal_E_xhat}) by $\mathbf{E}^+$ generates $\mathbf{E}^+\tilde{\mathbf{x}}=\hat{\mathbf{x}}$. In the RHB method, a relation $\hat{f}(\hat{\mathbf{x}})=\mathbf{E}^+\tilde{f}(\tilde{\mathbf{x}})$ is introduced to replace the original term, where $\tilde{f}(\tilde{\mathbf{x}})$ is the value of $f(\mathbf{x})$ at $M$ discrete time collocations $\tilde{\mathbf{x}}$. Therefore, Equation (\ref{con:HBForm}) can be rewritten as
\begin{equation}
	\omega\mathbf{AE^+}\tilde{\mathbf{x}}=\mathbf{E^+}\tilde{f}(\tilde{\mathbf{x}}).
	\label{con:RHBForm}
\end{equation}
Equation (\ref{con:RHBForm}) is the RHB algebraic system consisting of $2N+1$ nonlinear algebraic equations. As will be proven later, the algebraic system (\ref{con:RHBForm}) is strictly equivalent to the HB algebraic system when the number of collocation nodes $M>(\phi+1)N$. However, the system (\ref{con:RHBForm}) is expressed by time domain quantities and constant matrix, hence completely avoiding the tedious symbolic calculations.

\subsection{Dealiasing mechanism of RHB}
The RHB method, based on equivalently re-expressing the HB method with pure collocations, can completely eliminate the famous aliasing problem that occurs in the HDHB method (see Supplementary Part 2). Affected by aliasing errors, the HDHB method has a severe problem of non-physical solutions. Our previous study demonstrated that the HDHB method is not a variant of HB method, but a time domain collocation method in disguise\cite{dai2012simple}. Furthermore, the principle beneath the aliasing phenomenon has been illustrated as follows\cite{boyd2001chebyshev,dai2014time}:

\begin{quote}
	\emph{Rules of aliasing}: If the interval $\alpha\in[0,2\pi]$ is discretized in time domain with uniform grid spacing $h$, then the wavenumbers in the trigonometric interpolant lies on the range $n\in[-L,L]$ , where $L=\pi/h$ is the so-called "aliasing limit" wavenumber. Higher wavenumbers ($\lvert n \rvert>L$) will be aliased to lower wavenumbers $n_a$:
	\[
	n_a=n-2mL,
	\]
	where $n_a\in[-L,L]$, $m$ is an integer.
\end{quote}

According to \emph{rules of aliasing}, as long as the high order to be aliased are outside the frequency range $[-\phi N,\phi N]$, we can avoid aliasing.

\begin{thm}[Dealiasing]\label{thm1}
	Suppose the nonlinear differential equation (\ref{con:DifferentialEquation}) is in polynomial-type nonlinearity with a degree of nonlinearity $\phi$. Then, the aliasing phenomenon can be eliminated if the number of collocation nodes in the RHB method satisfies $M>(\phi+1)N$, where $N$ is the truncation order of the RHB method.
\end{thm}

\begin{proof}[Proof of Theorem~{\upshape\ref{thm1}}]
	For a linear system, viz., $\phi=1$, there is no aliasing. Therefore, $\phi\geq2$ is considered. The higher-order harmonics of RHB algebraic equations lie in the range $[-\phi N,-N)\cup(N,\phi N]$ due to the nonlinearity $\phi$. According to \emph{rules of aliasing}, higher-order harmonics will be aliased to the range of $[-\phi N-2mL,-N-2mL)\cup(N-2mL,\phi N-2mL]$. To avoid aliasing, the aliased range should not intersect with the low order range $[-N,N]$. There are three situations according to different $m$:
	\begin{enumerate}
		\item If $m=0$: there is obviously no intersection between interval $[-\phi N-2mL,-N-2mL)\cup(N-2mL,\phi N-2mL]$ and interval $[-N,N]$.
		
		\item if $m\geq1$: there is obviously no intersection between interval $[-\phi N-2mL,-N-2mL)$ and interval $[-N,N]$. To ensure  $(N-2mL,\phi N-2mL]$ does not intersect with $[-N,N]$, there must be $N-2mL>N$ or $\phi N-2mL<-N$. hence leading to $M>(\phi+1)N/m$. 
		
		\item If $m\leq-1$: it is similar to case 2.
	\end{enumerate}
	To sum up, $M>(\phi+1)N$ must be satisfied to ensure that $M>(\phi+1)N/m$ for any integer $m$.
\end{proof}

\subsection{Conditional equivalence of RHB and HB}
The previous section has explained that aliasing errors can be avoided when $M$ satisfies certain condition. In this section, we explain why the RHB method is equivalent to the HB method.

\begin{thm}[Conditional equivalence]\label{thm2}
	If number of collocations $M$, truncated order of harmonics $N$ and degree of nonlinearity $\phi$ satisfy $M>(\phi+1)N$, then the RHB method and the HB method are equivalent.
\end{thm}

\begin{proof}[Proof of Theorem~{\upshape\ref{thm2}}]
	Since $\tilde{\mathbf{x}}=\mathbf{E}\hat{\mathbf{x}}$ always holds, we just prove that $\mathbf{E^+}\tilde{f}(\tilde{\mathbf{x}})=\hat{f}(\hat{\mathbf{x}})$. $\hat{f}(\hat{\mathbf{x}})$ and $\tilde{f}(\tilde{\mathbf{x}})$ are assumed as $\hat{f}(\hat{\mathbf{x}})=[\mathbf{h}_0\ \ \mathbf{h}_1\ \ \cdots\ \ \mathbf{h}_{2N}]^\mathrm{T}$ and $\tilde{f}(\tilde{\mathbf{x}})=[f(\mathbf{x}(t_1))\ \ f(\mathbf{x}(t_2))\ \ \cdots\ \ f(\mathbf{x}(t_M))]^{\mathrm{T}}$, where $\mathbf{h}_0$, $\mathbf{h}_1$, $\cdots$, $\mathbf{h}_{2N}$ are the Fourier coefficients, and $f(\mathbf{x}(t_i))$ is the corresponding temporal quantity at prescribed time instant $t_i$, 
	\[
	f(\mathbf{x}(t_i))=\mathbf{h}_0+\sum_{n=1}^{\phi N}[\mathbf{h}_{2n-1}\cos(n\omega t_i)+\mathbf{h}_{2n}\sin(n\omega t_i)].
	\]
	
	For $\phi=1$, it is obvious that $\mathbf{E^+}\tilde{f}(\tilde{\mathbf{x}})=\hat{f}(\hat{\mathbf{x}})$ holds.
	
	For $\phi\geq2$, $\tilde{f}(\tilde{\mathbf{x}})$ can be expressed as
	\[
	\tilde{f}(\tilde{\mathbf{x}})=
	\left[
	\begin{array}{c}
		f(\mathbf{x}(t_1)) \\
		f(\mathbf{x}(t_2)) \\
		\vdots\\
		f(\mathbf{x}(t_M))
	\end{array}
	\right]
	=
	\left[
	\begin{array}{c}
		\mathbf{h}_0+\sum_{n=1}^{\phi N}[\mathbf{h}_{2n-1}\cos(n\omega 		t_1)+\mathbf{h}_{2n}\sin(n\omega t_1)]\\
		\mathbf{h}_0+\sum_{n=1}^{\phi N}[\mathbf{h}_{2n-1}\cos(n\omega 		t_2)+\mathbf{h}_{2n}\sin(n\omega t_2)]\\
		\vdots\\
		\mathbf{h}_0+\sum_{n=1}^{\phi N}[\mathbf{h}_{2n-1}\cos(n\omega 	t_M)+\mathbf{h}_{2n}\sin(n\omega t_M)]
	\end{array}
	\right].
	\]
	Divides the first $N$-order and $N+1\sim\phi N$ higher-order terms into two parts:
	\[
	\tilde{f}(\tilde{\mathbf{x}})=
	\left[
	\begin{array}{c}
		\mathbf{h}_0+\sum_{n=1}^{N}[\mathbf{h}_{2n-1}\cos(n\omega 	 t_1)+\mathbf{h}_{2n}\sin(n\omega t_1)]\\
		\mathbf{h}_0+\sum_{n=1}^{N}[\mathbf{h}_{2n-1}\cos(n\omega  t_2)+\mathbf{h}_{2n}\sin(n\omega 	t_2)]\\
		\vdots\vspace{2ex}\\
		\mathbf{h}_0+\sum_{n=1}^{N}[\mathbf{h}_{2n-1}\cos(n\omega 	 t_M)+\mathbf{h}_{2n}\sin(n\omega t_M)]
	\end{array}
	\right]
	\]
	\[
	+
	\left[
	\begin{array}{c}
		\sum_{n=N+1}^{\phi N}[\mathbf{h}_{2n-1}\cos(n\omega 	 t_1)+\mathbf{h}_{2n}\sin(n\omega t_1)]\\
		\sum_{n=N+1}^{\phi N}[\mathbf{h}_{2n-1}\cos(n\omega 	 t_2)+\mathbf{h}_{2n}\sin(n\omega t_2)]\\
		\vdots\\
		\sum_{n=N+1}^{\phi N}[\mathbf{h}_{2n-1}\cos(n\omega 	 t_M)+\mathbf{h}_{2n}\sin(n\omega t_M)]
	\end{array}
	\right].
	\]
	The first term of the above formula is $\mathbf{E}\hat{f}(\hat{\mathbf{x}})$, thus
	\[
	\tilde{f}(\tilde{\mathbf{x}})=\mathbf{E}\hat{f}(\hat{\mathbf{x}})+\mathbf{E}_1\hat{f}^{'}(\hat{\mathbf{x}}),
	\]
	where
	\[
	\mathbf{E}_1=
	\left[
	\begin{array}{ccccc}
		\cos(N+1)\omega t_1&\sin(N+1)\omega t_1&\cdots&\cos(\phi N)\omega 	t_1&\sin(\phi N)\omega t_1\\
		\cos(N+1)\omega t_2&\sin(N+1)\omega t_2&\cdots&\cos(\phi N)\omega 	t_2&\sin(\phi N)\omega t_2\\
		\vdots             &\vdots             &\ddots&\vdots                &\vdots                \\
		\cos(N+1)\omega t_M&\sin(N+1)\omega t_M&\cdots&\cos(\phi N)\omega 	t_M&\sin(\phi N)\omega t_M
	\end{array}
	\right],
	\]
	\[
	\hat{f}^{'}(\hat{\mathbf{x}})=
	\left[
	\begin{array}{c}
		\mathbf{h}_{2N+1}\\
		\mathbf{h}_{2N+2}\\
		\vdots\\
		\mathbf{h}_{\phi N}
	\end{array}
	\right].
	\]
	Therefore, $\mathbf{E}^+\tilde{f}(\tilde{\mathbf{x}})$ is
	\[
	\mathbf{E}^+\tilde{f}(\tilde{\mathbf{x}})=\mathbf{E}^+\mathbf{E}\hat{f}(\hat{\mathbf{x}})+\mathbf{E}^+\mathbf{E}_1\hat{f}^{'}(\hat{\mathbf{x}})=\hat{f}(\hat{\mathbf{x}})+\mathbf{E}^+\mathbf{E}_1\hat{f}^{'}(\hat{\mathbf{x}}).
	\]
	We define $\mathbf{E}_{\mathrm{A}}=\mathbf{E}^+\mathbf{E}_1$ as "aliasing matrix". Next, we investigate the effects of the number of collocations $M$ on the aliasing matrix $\mathbf{E}_{\mathrm{A}}$. In details, the elements of row $i$ and column $j$ of the aliasing matrix $\mathbf{E}_{\mathrm{A}}$ are shown in  Table \ref{tab1}.
	
	\begin{table}[h]
		\begin{center}
			\begin{minipage}{\textwidth}
				\caption{Explicit expression of the elements in aliasing matrix $\mathbf{E}_{\mathrm{A}}$}\label{tab1}
				\begin{tabular*}{\textwidth}{@{\extracolsep{\fill}}ccc@{\extracolsep{\fill}}}
					\toprule%
					$j$&$i$& Elements\\
					\midrule
					&$i=1$&$\dfrac{1}{M}\sum\limits_{n=1}^{M}\cos(N+\dfrac{j+1}{2})\omega t_n$\\
					$j$\ (odd)&$i\geq2$ (odd)&$\dfrac{2}{M}\sum\limits_{n=1}^{M}\sin\dfrac{i-1}{2}\omega t_n\cos(N+\dfrac{j+1}{2})\omega t_n$\\
					&$i\geq2$ (even)&$\dfrac{2}{M}\sum\limits_{n=1}^{M}\cos\dfrac{i}{2}\omega t_n\cos(N+\dfrac{j+1}{2})\omega t_n$\\
					\midrule
					&$i=1$&$\dfrac{1}{M}\sum\limits_{n=1}^{M}\sin(N+\dfrac{j+1}{2})\omega t_n$\\
					$j$ (even)&$i\geq2$ (odd)&$\dfrac{2}{M}\sum\limits_{n=1}^{M}\sin\dfrac{i-1}{2}\omega t_n\sin(N+\dfrac{j}{2})\omega t_n$\\
					&$i\geq2$ (even)&$\dfrac{2}{M}\sum\limits_{n=1}^{M}\cos\dfrac{i}{2}\omega t_n\sin(N+\dfrac{j}{2})\omega t_n$\\
					\bottomrule
				\end{tabular*}
			\end{minipage}
		\end{center}
	\end{table}
	The elements in the aliasing matrix $\mathbf{E}_{\mathrm{A}}$ are discussed as follows:
	\begin{enumerate}
		\item $\dfrac{1}{M}\sum_{n=1}^{M}\cos(N+\dfrac{j+1}{2})\omega t_n$: This term is equal to 1 when $i+j=2(kM-N)-1$ and $M\leq\phi N$ hold, where $k$ is any positive integer satisfying $kM\leq\phi N$. This term is always zero if $M>\phi N$.
		
		\item $\dfrac{1}{M}\sum_{n=1}^{M}\sin(N+\dfrac{j}{2})\omega t_n$: Regardless of the value of $M$, this term is always zero.
		
		\item $\dfrac{2}{M}\sum_{n=1}^{M}\sin\dfrac{i-1}{2}\omega t_n\cos(N+\dfrac{j+1}{2})\omega t_n$: Regardless of the value of $M$, this term is always zero.
		
		\item $\dfrac{2}{M}\sum_{n=1}^{M}\sin\dfrac{i-1}{2}\omega t_n\sin(N+\dfrac{j}{2})\omega t_n$: This term is equal to -1 when $i+j=2(kM-N)+1$ and $M\leq(\phi+1)N$ hold, where $k$ is any positive integer satisfying $kM\leq(\phi+1)N$. In addition, if $M\leq\phi N-1$, this term will be 1 when $j-i=2(kM-N)-1$ holds. This term is always zero if $M>(\phi+1)N$.
		
		\item $\dfrac{2}{M}\sum_{n=1}^{M}\cos\dfrac{i}{2}\omega t_n\cos(N+\dfrac{j+1}{2})\omega t_n$: This term is equal to 1 when $i+j=2(kM-N)-1$ and $M\leq(\phi+1)N$ hold, where $k$ is any positive integer satisfying $kM\leq(\phi+1)N$. In addition, if $M\leq\phi N-1$, this term will be 1 when $j-i=2(kM-N)-1$ holds. This term is always zero if $M>(\phi+1)N$.
		
		\item $\dfrac{2}{M}\sum_{n=1}^{M}\cos\dfrac{i}{2}\omega t_n\sin(N+\dfrac{j}{2})\omega t_n$: Regardless of the value of $M$, this term is always zero.
	\end{enumerate}
	To sum up, all elements in the aliasing matrix will be zero if $M>(\phi+1)N$. This means the aliasing matrix $\mathbf{E}_{\mathrm{A}}$ is a zeros matrix, so that
	\[
	\mathbf{E}^+\tilde{f}(\tilde{\mathbf{x}})=\hat{f}(\hat{\mathbf{x}})+\mathbf{E}_{\mathrm{A}}\hat{f}^{'}(\hat{\mathbf{x}})=\hat{f}(\hat{\mathbf{x}}).
	\]
	The conditional equivalence of $\mathbf{E}^+\tilde{f}(\tilde{\mathbf{x}})$ and $\hat{f}(\hat{\mathbf{x}})$ is proven, so the RHB has been shown to completely reconstruct the HB method by time domain quantities.
\end{proof}

It is emphasized that the conditional identity $\mathbf{E}^+\tilde{f}(\tilde{\mathbf{x}})=\hat{f}(\hat{\mathbf{x}})\ \ (M>(\phi+1)N)$ is a newly discovered identity in this study, which analytically links the nonlinear Fourier components in the time domain and the frequency domain. This conditional identity may find applications in a wide range of problems where Fourier analysis is employed.

\section{Results}
\subsection{Nonlinear Duffing equation}
\begin{figure}[thbp]%
	\centering
	\includegraphics[width=1\textwidth]{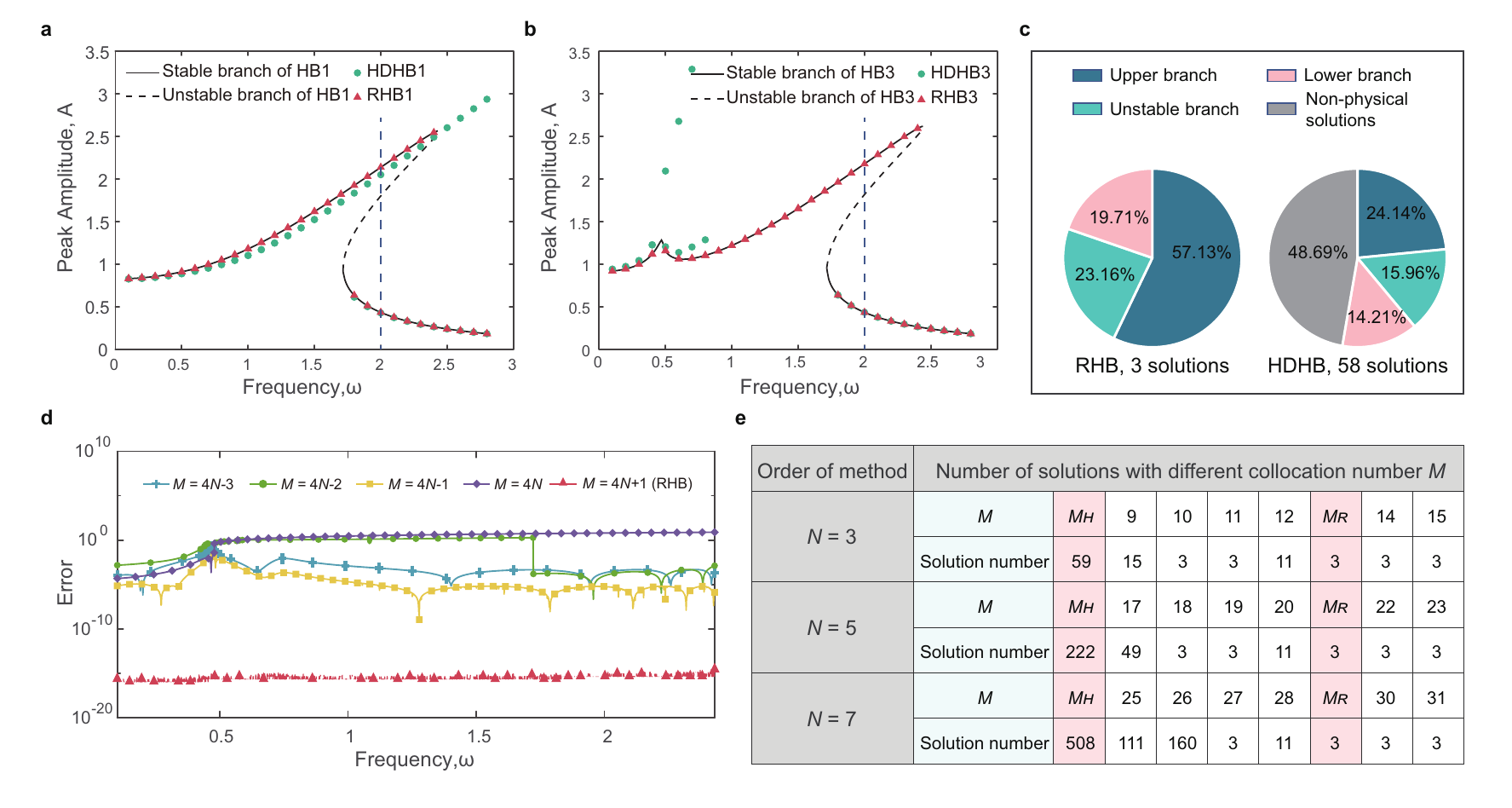}
	\caption{\textbf{Periodic analysis of Duffing equation using the RHB and the HB methods.} The capability of dealiasing and the high accuracy of the RHB method are verified in solving Duffing equation. \textbf{a}, \textbf{b}, Amplitude-frequency response curves for $N=1$ and $N=3$, respectively. \textbf{c}, Statistical distribution of solutions by the RHB3 and the HB3 methods for $\omega=2$. \textbf{d}, Error curves corresponding to solutions with different numbers of collocations. \textbf{e}, Number of solutions for the RHB method with different orders and number of collocations.}
	\label{fig:Fig2}
\end{figure}

The Duffing oscillator is a classical nonlinear dynamical system which is frequently encountered in mechanical vibrations and electrical circuits\cite{lai2005inducing}. We choose the Duffing equation $\ddot{x}+c\dot{x}+kx+\alpha x^{\phi}=F\sin(\omega t)$ as a prototypical example to elucidate the principle and advantages of the presented RHB method, where $c$, $k$, $\alpha$ and $\phi$ are the damping, linear, nonlinear coefficients, and the degree of nonlinearity, respectively; $F$ and $\omega$ are the forcing amplitude and frequency, respectively. We compare the computational accuracy of the RHB method and the HDHB method against the benchmark HB method (see Supplementary Part 5) in terms of the amplitude-frequency curves. We can see from Figs. \ref{fig:Fig2}\textbf{a} and \ref{fig:Fig2}\textbf{b} that the RHB results coincide with the HB results, whereas the HDHB results deviate from the benchmark.

The Monte Carlo simulation is implemented to check the generation of non-physical solutions (see Fig. \ref{fig:Fig2}\textbf{c}). We carry out 10000 computations starting with initial values randomly generated from $[-5, 5]$ for the specified excitation frequency $\omega=2$. Fig. \ref{fig:Fig2}\textbf{c} shows that the RHB method produces 3 physical solutions. In contrast, the HDHB method produces 58 different solutions, 55 of which are non-physical ones. In particular, the probabilities of obtaining a physical solution for RHB is 100\% while that for HDHB is 51.3\%, indicating that the RHB method does not produce non-physical solutions. We get similar results when solving other Duffing equations with mixes of nonlinear terms (see Supplementary Part 6).

\begin{figure}[thbp]%
	\centering
	\includegraphics[width=1\textwidth]{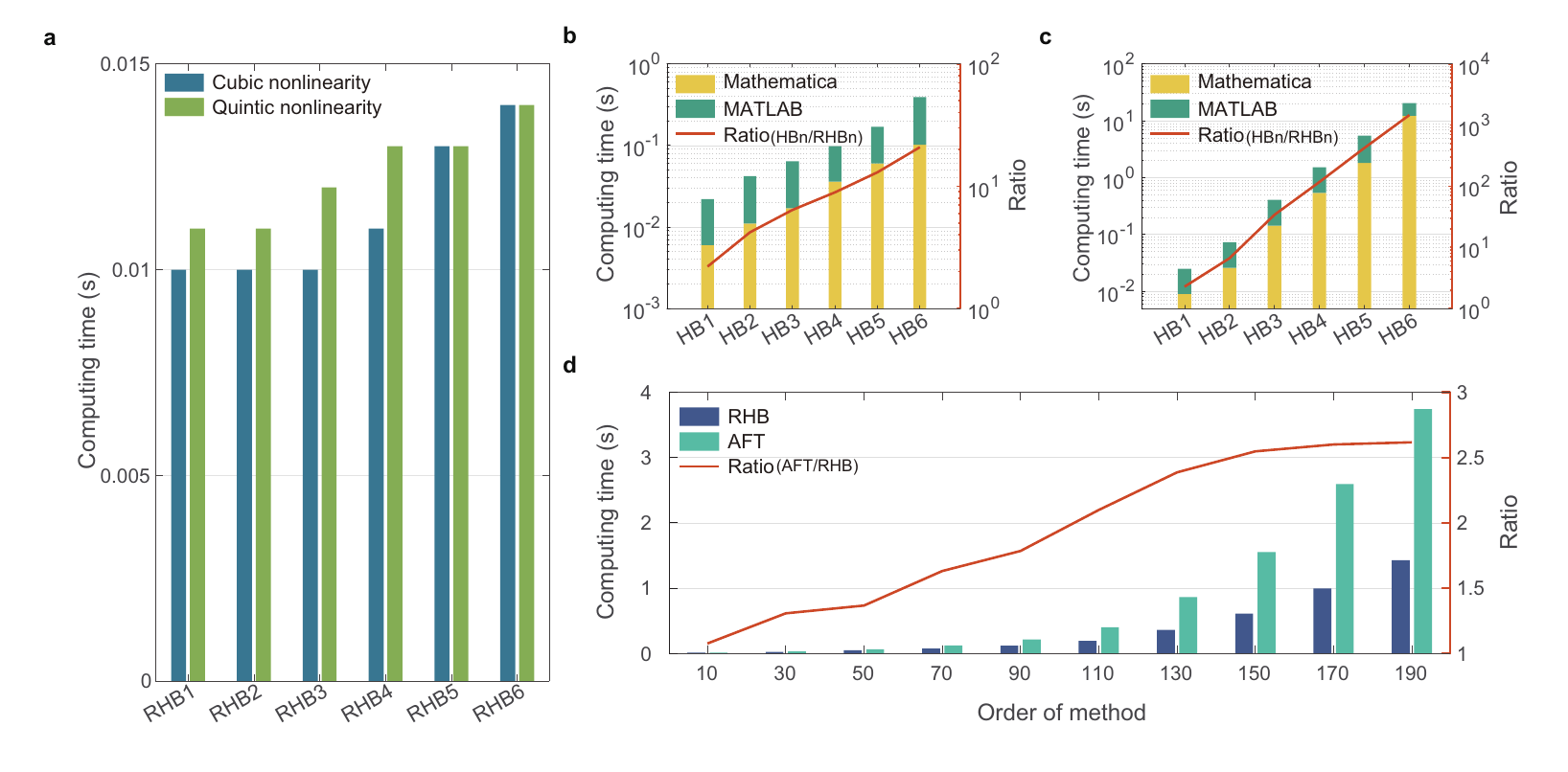}
	\caption{\textbf{Comparison of computational efficiency of the HB, the AFT and the RHB methods.} \textbf{a}, Computing time of the RHB method for solving Duffing equations with cubic and quintic nonlinear terms. \textbf{b}, \textbf{c}, Computing time of the HB method and the ratio to the corresponding RHB method for solving Duffing equations with (\textbf{b}) cubic nonlinearity and (\textbf{c}) quintic nonlinearity. \textbf{d}, Computing time of the AFT method and its ratio to the RHB method when dealing with cubic Duffing equation.}
	\label{fig:Fig3}
\end{figure}

Next, we explore how the number of collocations affects the accuracy of RHB method (choose $N=3$ as an example). Fig. \ref{fig:Fig2}\textbf{d} presents the computational errors of amplitude-frequency curve against the HB3 result. It shows that the RHB method will produce significantly large errors if the minimum number of collocations, i.e., $M_{\mathrm{R}}=(\phi+1)N+1$, is not satisfied. Once $M\geq M_{\mathrm{R}}$, the RHB method produces almost the same result as the HB method (the missing portion of the red curve suggests a near-zero error). For further analysis, we carry out the Monte Carlo simulation at $\omega=2$  for the RHB method with different number of collocations. It is clearly shown from Fig. \ref{fig:Fig2}\textbf{d} that the RHB solutions are confined to the three real ones once the minimum collocations are satisfied, otherwise non-physical solutions may occur. Although for some $M<M_{\mathrm{R}}$ the number of solutions is correct, the solution accuracy is several orders of magnitude lower (see Fig. \ref{fig:Fig2}\textbf{d}).

Then, we compare the computational efficiency of the RHB, HB, and AFT methods for solving Duffing equations with cubic and quintic nonlinearities (see Supplementary Part 7 for details). In Fig. \ref{fig:Fig3}, each data is obtained by averaging one hundred runs. Fig. \ref{fig:Fig3}\textbf{a} shows that the computing time only slightly increases with the increase of the order of method, viz. order of harmonics retained, in the RHB method. Besides, the computational effort does not notably increase with the degree of nonlinearity. By contrast, however, even with the aid of computer algebra software the computing time of HB method increases exponentially with order, leading to an exponentially increasing ratio between the computing time of HB and RHB methods (see Fig. \ref{fig:Fig3}\textbf{b}-\ref{fig:Fig3}\textbf{c}). Even for a six-order case, the RHB is more than three orders of magnitude faster than the HB (see Fig. \ref{fig:Fig3}\textbf{c}). By alternatively transforming harmonic balancing process between frequency domain and time domain, the AFT method speeds up the computations. However, it is still much slower than the proposed RHB method. As shown in Fig. \ref{fig:Fig3}\textbf{d}, the RHB method could be more than two times faster than the AFT method once high order harmonics are involved.

\subsection{Cavitation bubble dynamics}
The Rayleigh-Plesset (R-P) equation describes the motion of a spherical cavitation bubble in incompressible liquid. As a fundamental problem, it receives extensive attentions\cite{marmottant2003controlled, lajoinie2018non, menzl2016molecular, lohse2016homogeneous}. The R-P equation is written as $R\dot{R}=-3\dot{R}^2/2-A\dot{R}/R-B/R+C/R^3+D-E\cos(\omega t)$, where $R$ is the radius of bubble, $\omega$ is the frequency of external excitation, $t$ is the time, and $A-E$ are system parameters (see Supplementary Part 8). Apparently, the R-P equation involves non-polynomial nonlinearity, so it cannot be solved by the HB method. Although the HDHB and AFT methods can solve this problem, they will produce inevitable aliasing errors\cite{liu2006comparison, krack2019harmonic}. Herein, we propose to use a simple recasting technique that can equivalently convert a large class of non-polynomial systems into polynomial types\cite{cochelin2009high, karkar2013high}, so that the RHB method can solve the problem without aliasing. In particular, the present R-P equation is converted to a fourth order polynomial system by recasting (see Supplementary Part 8). Fig. \ref{fig:Fig4}\textbf{a} presents the time response curves of the bubble radius and its changing rate in a whole period. The 250-th order RHB method, i.e., RHB250, was used to accurately capture the expansion ($20 \mu s - 32 \mu s$), the collapse ($32 \mu s - 39 \mu s$), and the rebound ($39\mu s - 60 \mu s, 0 \mu s - 20 \mu s$) stages. First, as the outer pressure decreases, the bubble transits from the rebound of last period to the initial expansion, and approaches the maximum radius at around $32 \mu s$. Then the bubble collapses swiftly during around $7 \mu s$. After that, the bubble reenters the rebound stage as the outer pressure increases, and dissipates energy by damped vibrations.

\begin{figure}[t]
	\centering
	\includegraphics[width=0.55\textwidth]{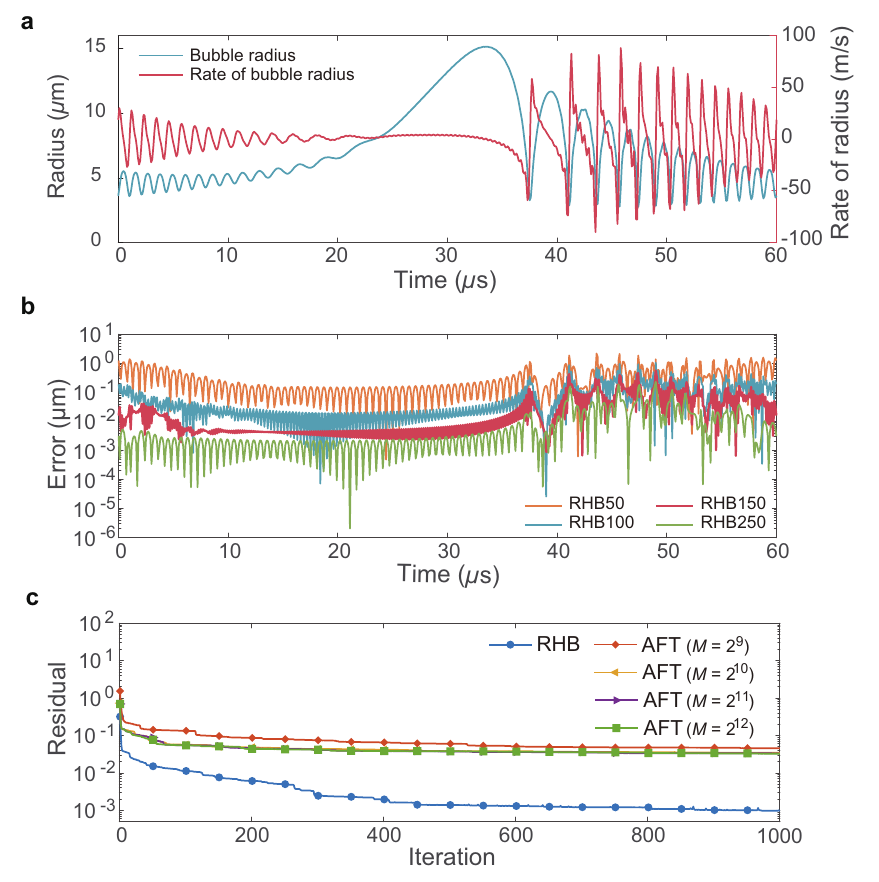}
	\caption{\textbf{Results of the Rayleigh-Plesset equation obtained by the RHB and the AFT method.}  \textbf{a}, Time domain response curves and rates of cavity radius obtained by the RHB method with $N=250$. \textbf{b}, Error curves of the RHB method with different orders against the benchmark numerical result. \textbf{c}, Residual convergence history of resultant algebraic equations of AFT and RHB with $N=150$.}
	\label{fig:Fig4}
\end{figure}

The accuracy of the RHB with different orders is explored in Fig. \ref{fig:Fig4}\textbf{b} with a highly accurate numerical integration method as the benchmark. It shows that the computational error significantly decreases as $N$ increases to 250, implying that the present system contains high frequency responses and thus very higher order harmonics are necessary. Moreover, we observe that the computational accuracy is relatively low in the rebound stage, due to the bubble’s vibration at its natural frequency $\omega_0$, while the basis harmonics are of the integer multiples of frequency $\omega$ that is noncompatible with $\omega_0$. The accuracy of the recasting RHB and the AFT is compared in Fig. \ref{fig:Fig4}\textbf{c}. It shows that the RHB15 converges much faster than the AFT15. The RHB result gradually decreases to $10^{-3}$, which is two orders of magnitude more accurate than the AFT. It is worth to mention that a novel Jacobian-inverse-free method developed by our group\cite{dai2014solutions} is used to solve the RHB and AFT resultant NAEs. If the classical Newton-Raphson method is used, the RHB converges to $10^{-12}$, but the AFT diverges (see Supplementary Part 8). It is concluded that the aliasing induced by non-polynomial nonlinearity is the main cause for the AFT’s poor performance in solving the R-P equation. In sum, the recasting RHB method is superior to the AFT method in both computational accuracy and efficiency when dealing with non-polynomial nonlinear systems.

\subsection{Three-body problem}
In this part, we evaluate the performance of the RHB method by the renowned Circular Restricted Three-Body Problem (CRTBP). The CRTBP is a six-dimensional non-polynomial nonlinear system (see Supplementary Part 10) whose periodic orbits are of vital importance to deep space exploration\cite{koon2000dynamical, stone2019statistical}. So far, the most widely used method for this problem is the differential correction (DC) method; however, the DC method only finds a nearly periodic solution in a recursive shooting manner (see Supplementary Part 11). Moreover, it cannot seek the periodic orbit by prescribing a desired period which is a crucial design factor. In contrast, the present RHB method provides semi-analytical solution that is inherently periodic with a prescribed period. Thus, the RHB method promises to be a powerful tool in designing CRTBP periodic orbits.

\begin{figure}[t]%
	\centering
	\includegraphics[width=1\textwidth]{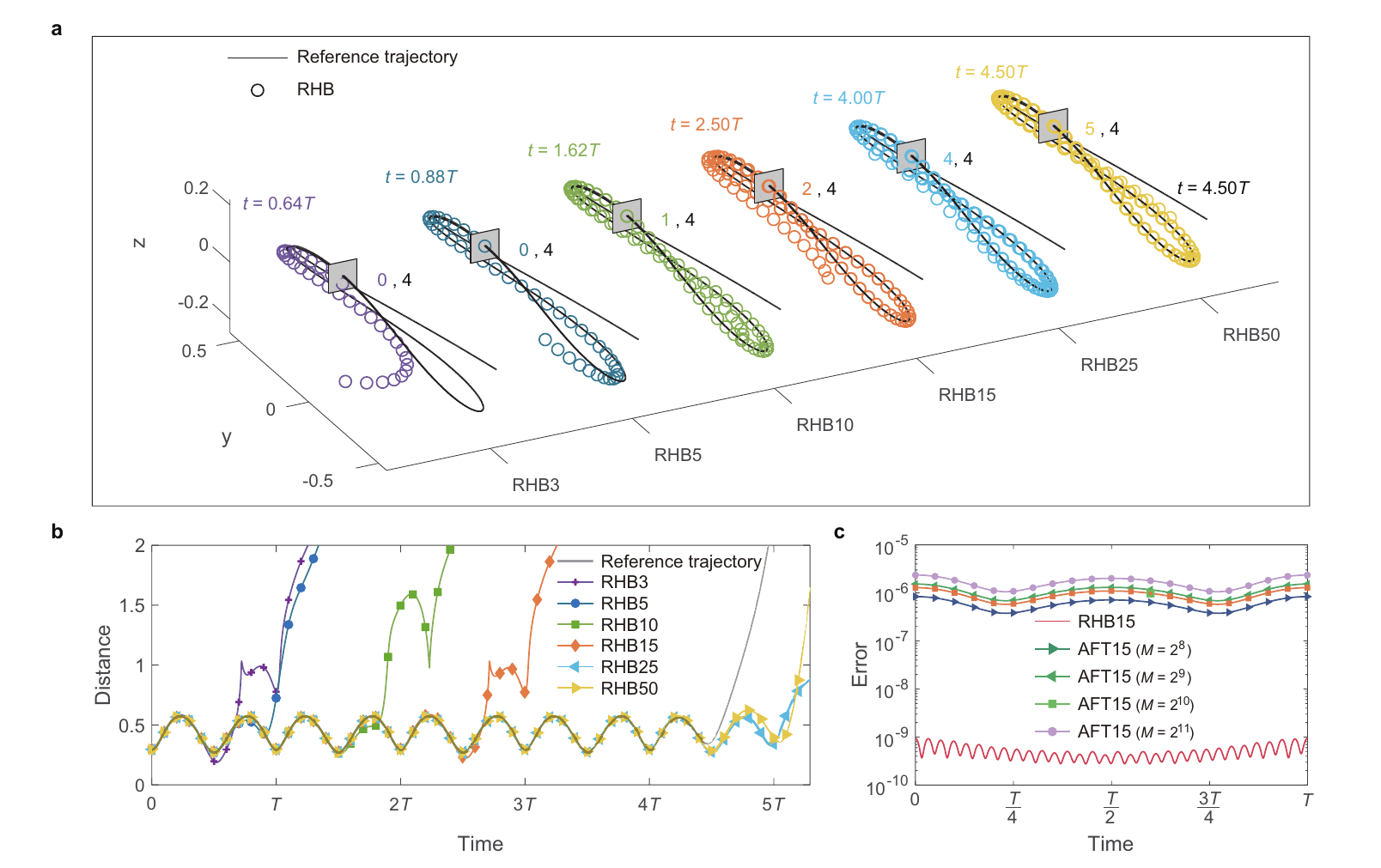}
	\caption{\textbf{Computational results of halo orbit in circular restricted three-body problem.} \textbf{a}, Trajectories obtained by the RHB method with different number of harmonics, in comparison with the trajectory obtained by highly accurate numerical integration of halo orbits. Using initial condition provided by the RHB solution, we recorded the number of periods that can be maintained without control near the initial position (as marked by a grey section) of the orbit. \textbf{b}, The change of the distance between different trajectories and the $L_2$ point. \textbf{c}, Error curves of the RHB and the AFT methods with $N=15$.}
	\label{fig:Fig5}
\end{figure}

Targeting at the Earth-Moon $L_2$ point, the computational results for the halo orbits are shown in Fig. \ref{fig:Fig5}. Using the RHB method, we reproduce the halo orbit of the Queqiao relay satellite for the Chang’e 4 Lunnar farside mission (Fig. \ref{fig:Fig5}\textbf{a}). Concretely, the RHB is used to produce a highly periodic nominal orbit, whose accuracy is verified via numerical orbit propagation. Fig. \ref{fig:Fig5}\textbf{a} presents the numerically propagated orbits obtained with initial states from the RHB with different orders. To evaluate orbit-keeping ability, we set a small section on the $xOz$ plane. The number of times that the orbit passes through the section stands for the number of whole periods before the orbit deviates. It is seen that the DC method, at its highest accuracy, provides a reference trajectory that drifts away after 4.5$T$, with $T$ being the period of the halo orbit, and the higher order RHB method shows a better orbit-keeping property, leaving more intersection points on the section. The RHB50 solution outperforms the best DC solution as is investigated in details in Fig. 5b. Noting that the halo orbit is inherently unstable\cite{srivastava2018halo, gomez1993study}, an arbitrary small disturbance may cause a significant discrepancy. Nevertheless, Figs. \ref{fig:Fig5}\textbf{a} and \ref{fig:Fig5}\textbf{b} have revealed the RHB’s high accuracy. The comparison of computational accuracy of the RHB method with recasting technique and the AFT method is provided in Fig. \ref{fig:Fig5}\textbf{c}. In particular, the RHB15 is compared with the AFT15 with different sampling rates against the benchmark result of RHB50. Numerical result suggests that the present method is about three orders of magnitude more accurate than the AFT method, regardless of the sampling rate adopted.

Furthermore, we use the RHB method in conjunction with the parameter-sweeping approach to capture the amplitude-frequency response curve of the CRTBP (Fig. \ref{fig:Fig6}\textbf{a}). The range of orbital frequency is prescribed as $[1, 3.5]$. The approximate solution of simplified Richardson model (see Supplementary Part 12) is used to supply the RHB method with initials. Four branches in Fig. \ref{fig:Fig6}\textbf{a} are corresponding to the distant retrograde orbit family (Fig. \ref{fig:Fig6}\textbf{b}), the vertical Lyapunov orbit family (Fig. \ref{fig:Fig6}\textbf{c}), the planar Lyapunov orbit family (Fig. \ref{fig:Fig6}\textbf{d}) and the halo orbit family (Fig. \ref{fig:Fig6}\textbf{d}). Actually, the present method can capture more families of periodic orbits (see Supplementary Part 13), although they are not the major concerns of current space missions. Moreover, we notice from Fig. \ref{fig:Fig6}\textbf{a} that the planar Lyapunov orbit family rendezvouses with the halo orbit family, indicating the existence of bifurcation point at $\omega=1.84$. Thus, it is shown that the present method can capture a rich variety of orbit families near $L_2$ and simultaneously pinpoint their frequency ranges.

\begin{figure}[t]%
	\centering
	\includegraphics[width=1\textwidth]{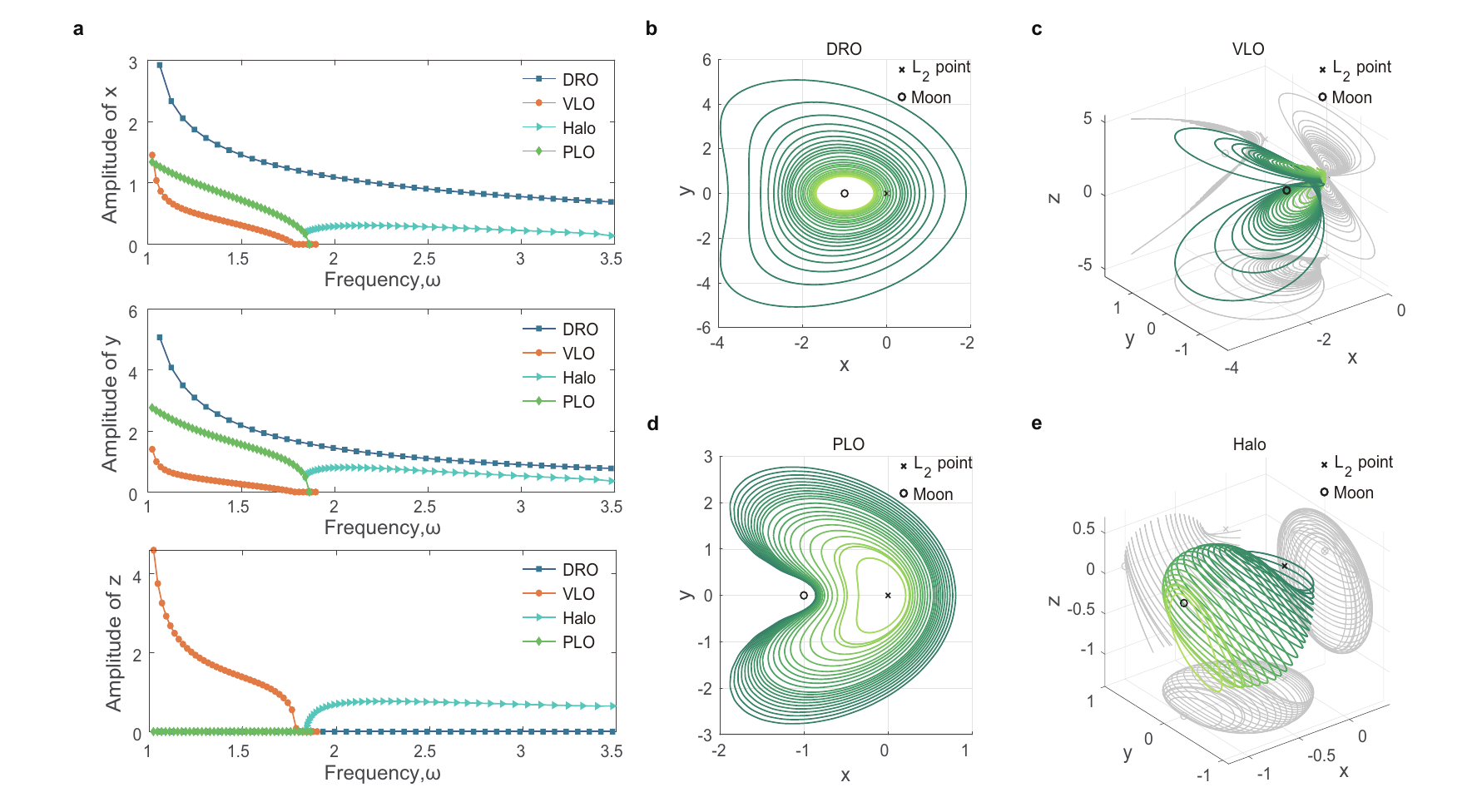}
	\caption{\textbf{Amplitude-frequency response curves of the CRTBP system and its corresponding orbit families.} Combined with frequency-sweeping process, the RHB method provides a novel strategy to periodic orbit design in three body system. \textbf{a}, Amplitude-frequency response curves of $x$, $y$, $z$ directions in the frequency range $[1, 3.5]$. Four curves are shown here. Each curve corresponds to a family of orbits. \textbf{b}, Distant retrograde orbit family. \textbf{c}, Vertical Lyapunov orbit family. \textbf{d}, Planar Lyapunov orbit family. \textbf{e}, Halo orbit family.}
	\label{fig:Fig6}
\end{figure}

\section{Discussion}
Constrained by either high computational cost or potential aliasing, current computational methods for periodic analysis of complex nonlinear dynamical systems are insufficient to deal with multi-disciplinary problems generally and robustly. Our study reveals that the fundamental problem lies in the biased cognition of harmonic balancing process in time domain. In this study we propose a collocation-based harmonic balancing framework to successfully incorporate the HB, the AFT, and the HDHB methods into a unity, which is built on a newly discovered conditional identity by introducing an aliasing matrix.  Based on this framework the powerful RHB method is proposed through letting the aliasing matrix be zero. In solving polynomial nonlinear problems, we demonstrate that the RHB method transforms into the classical HB method ($M=(\phi+1)N+1$), the HDHB method ($M=2N+1$), and the AFT method $(M=2\phi N+1)$ upon manipulating the collocation number. For the more general but tricky non-polynomial nonlinear problems, the RHB method combined with the recasting technique can solve highly accurate solutions without aliasing error. 

First, the prototypical Duffing equation is employed to elucidate the present method. We demonstrate that the RHB method completely avoids the symbolic-calculation-explosion belonging to the conventional high-order HB method. Thus, extremely high order solution can be obtained with little computational effort. We report that even in the framework of simple quintic Duffing equation, the RHB6 achieves more than 1460 times acceleration to the HB6 (Due to limited computer RAM utmost order HB6 is adopted. For higher order, the RHB’s advantage is more obvious). To compare with the HDHB method, statistical analysis of the distribution of solutions is conducted in solving the cubic Duffing equation. We find that the RHB produces all three physical solutions, while the HDHB method produces 55 additional non-physical solutions besides the physical ones. In addition, the advantage of the RHB method over the AFT method is illustrated. Oversampling of the AFT method is conquered by optimally choosing collocation nodes in the RHB, and property of aliasing matrix is demonstrated theoretically and numerically. Being an optimal reconstruction of the HB method, the RHB method is more than two times faster than the AFT method.

Second, the cavitation bubble problem is explored. Besides the high efficiency, the RHB method can be readily used to handle systems with complicated nonlinearity other than the polynomial-type which is strictly required by the HB method. In literature to date, all existing methods cannot generate aliasing-error-free solutions to the non-polynomial nonlinear systems. Here, we adopt the recasting technique to equivalently convert the original system to polynomial type, which is well suited for the proposed RHB method. We exactly obtain the periodic responses of the R-P equation, covering the expansion, the collapse, and the rebound stages, by using very high order RHB method. It is found that as many as 250 harmonics are required to capture the high frequency responses of the R-P equation with computational error of around $10^{-3} \mu m$. The RHB method in conjunction with the recasting technique is much more efficient and accurate than the AFT method, which can only achieve precision of $10^{-1} \mu m$ with sampling rate up to $2^{12}$.

Third, the famous CRTBP model is selected to further validate the present method. Since the present method facilitates high order computations, we employ the RHB50 to find the CRTBP’s periodic solution, i.e., $L_2$ halo orbit. We report that the RHB50 produces highly accurate orbit which is superior to the state-of-the-art differential-correction result. Furthermore, the $L_2$ halo orbit computed by the present RHB method is validated by the real flight data of China’s Queqiao satellite\cite{castelvecchi2018chinese} with the position error less than 8\%. Therefore, the RHB method promises to be an alternative efficient tool in the stage of the reference Halo orbit design. Moreover, the performance of the present method is compared with the AFT method. Specifically, the RHB15 in conjunction with the recasting technique is exploited to solve the CRTBP, and the result is three orders of magnitude more accurate than that of the AFT15. What is intriguing is that we easily obtained five families of CRTBP $L_2$ periodic orbits by means of the RHB method in conjunction with a straightforward parameter-sweeping procedure, opening a new path towards the CRTBP periodic orbit design. In summary, the present method is very efficient and easy-to-implement, and can find immediate applications in a variety of disciplines where highly accurate periodic solutions are wanted.

The present method can be further studied in two directions. First, the accuracy of the present method shows a decreasing effect when extremely high order harmonics are included. A similar phenomenon has been reported in the AFT method, and being attributed to the ill-conditionedness of the resultant nonlinear algebraic equations\cite{krack2019harmonic}. Hence, advanced preconditioning technology is a promising way to improve the present method. Second, once the periodic response of the nonlinear system consists of incommensurable frequencies, the present version of RHB method has to include a great many harmonics to approach the real response. It can be improved for the specific case with incommensurable excitation frequencies by borrowing the concept proposed in the supplemental-frequency harmonic balance method\cite{li2021supplemental}. 

\section*{Acknowledgement}
This work was supported by National Natural Science Foundation of China (No. 12072270, U2013206) and National Key Research and Development Program of China (No. 2021YFA0717100). Additional support was provided by Beijing Aerospace Control Center. We are also grateful to Ruilong Li and Jiye Zhang for helpful insights and discussions.


\newpage


\section*{Supplementary material}

# *Supplementary information of* Collocation-based harmonic balance framework for highly accurate periodic solution of nonlinear dynamical system

Honghua Dai[*,†,a,b], Zipu Yan[†,a,b], Xuechuan Wang[a,b,c], Xiaokui Yue[a,b], and Satya N. Atluri[c]

[a]*School of Astronautics, Northwestern Polytechnical University, 127 West Youyi Road, Xi'an, 710072, Shaanxi, China*
[b]*National Key Laboratory of Aerospace Flight Dynamics, Northwestern Polytechnical University, 127 West Youyi Road, Xi'an, 710072, Shaanxi, China*
[c]*Department of Mechanical Engineering, Texas Tech University, 2500 Broadway, Lubbock, 79401, Texas, USA*

March 13, 2022

---

[*]Corresponding author: E-mail: hhdai@nwpu.edu.cn
[†]These authors contributed equally to this work.

# Part 1. Classical harmonic balance method

Consider a general nonlinear dynamical system described by

$$\dot{\mathbf{x}} = f(\mathbf{x}, t), \tag{1}$$

where $f$ is the nonlinear function of the state variable $\mathbf{x}$ and $t$ represents explicit time variable. In the classical harmonic balance (HB) method, the approximate periodic solution $\mathbf{x}$ to the system (1) is assumed as a truncated Fourier series [1], that is

$$\mathbf{x} = \mathbf{x}_0 + \sum_{n=1}^{N}[\mathbf{x}_{2n-1}\cos(n\omega t) + \mathbf{x}_{2n}\sin(n\omega t)], \tag{2}$$

where $\mathbf{x}_1, \mathbf{x}_2, \cdots, \mathbf{x}_{2n}$ are unknown Fourier coefficients and $N$ is the number of harmonics retained, viz. the order of the harmonic balance method employed. The first derivative of equation (2) with respect to time is

$$\dot{\mathbf{x}} = \sum_{n=1}^{N}[-n\omega \mathbf{x}_{2n-1}\sin(n\omega t) + n\omega \mathbf{x}_{2n}\cos(n\omega t)]. \tag{3}$$

Let $\mathbf{S}$ be the row vector composed of the sinusoids:

$$\mathbf{S} = [1 \quad \cos(\omega t) \quad \sin(\omega t) \quad \cdots \quad \cos(N\omega t) \quad \sin(N\omega t)].$$

Thus,

$$\dot{\mathbf{S}} = [0 \quad -\omega\sin(\omega t) \quad \omega\cos(\omega t) \quad \cdots \quad -N\omega\sin(N\omega t) \quad N\omega\cos(N\omega t)] = \omega \mathbf{S}\mathbf{A},$$

where

$$\mathbf{A} = \begin{bmatrix} 0 & 0 & 0 & \cdots & 0 \\ 0 & \mathbf{J}_1 & 0 & \cdots & 0 \\ 0 & 0 & \mathbf{J}_2 & \cdots & 0 \\ \vdots & \vdots & \vdots & \ddots & \vdots \\ 0 & 0 & 0 & \cdots & \mathbf{J}_N \end{bmatrix}, \quad \mathbf{J}_n = n\begin{bmatrix} 0 & 1 \\ -1 & 0 \end{bmatrix}.$$

Therefore, equations (2) and (3) can be expressed as

$$\begin{cases} \mathbf{x} = \mathbf{S}\hat{\mathbf{x}} \\ \dot{\mathbf{x}} = \omega \mathbf{S}\mathbf{A}\hat{\mathbf{x}} \end{cases}, \tag{4}$$

where $\hat{\mathbf{x}} = [\mathbf{x}_0 \quad \mathbf{x}_1 \quad \cdots \quad \mathbf{x}_{2n}]^{\mathrm{T}}$. Substituting equation (4) into equation (1) yields

$$\omega \mathbf{S}\mathbf{A}\hat{\mathbf{x}} = \mathbf{S}\hat{f}(\hat{\mathbf{x}}), \tag{5}$$

where $\hat{f}$ is the nonlinear function of $\hat{\mathbf{x}}$. Upon using harmonic balancing technique, equation (5) transforms into the resultant nonlinear algebraic system, i.e. the so-called HB algebraic system

$$\omega \mathbf{A}\hat{\mathbf{x}} = \hat{f}(\hat{\mathbf{x}}). \tag{6}$$

Essentially, the $N$-th order HB algebraic system contains $2N+1$ nonlinear algebraic equations (NAEs), and the HB approximate solution $\hat{\mathbf{x}}$ can be readily obtained by

solving the resultant NAEs through NAE solvers. It is widely recognized that the HB method cannot handle the non-polynomial type nonlinear systems [2], and the HB method is strictly confined to low order computations due to the exponentially increasing complexity of $N$ obtaining the expression of $\hat{f}(\hat{\mathbf{x}})$. Even with the help of computer algebra software the computational complexity brought by moderately large harmonic balance analysis (a dozen of harmonics) is still unaffordable.

# Part 2. High-dimensional harmonic balance method

In order to avoid the tedious symbolic calculations involved in the HB method, the high-dimensional harmonic balance (HDHB) method was developed [3]. The main idea of the HDHB is to transform the HB algebraic system, i.e., equation (6), into corresponding time domain quantities. Firstly, the $2N + 1$ Fourier coefficients are expressed by $2N + 1$ equidistant discrete temporal quantities over one period via a constant collocation matirx $\mathbf{E}$:

$$\tilde{\mathbf{x}} = \mathbf{E}\hat{\mathbf{x}}, \tag{7}$$

wherein the vector for the temporal quantities is $\tilde{\mathbf{x}} = [\mathbf{x}(t_1) \ \mathbf{x}(t_2) \ \cdots \ \mathbf{x}(t_{2N+1})]^{\mathrm{T}}$, and the collocation matrix $\mathbf{E}$ is

$$\mathbf{E} = \begin{bmatrix} 1 & \cos(\omega t_1) & \sin(\omega t_1) & \cdots & \cos(N\omega t_1) & \sin(N\omega t_1) \\ 1 & \cos(\omega t_2) & \sin(\omega t_2) & \cdots & \cos(N\omega t_2) & \sin(N\omega t_2) \\ \vdots & \vdots & \vdots & \ddots & \vdots & \vdots \\ 1 & \cos(\omega t_{2N+1}) & \sin(\omega t_{2N+1}) & \cdots & \cos(N\omega t_{2N+1}) & \sin(N\omega t_{2N+1}) \end{bmatrix}.$$

Multiplying both sides of equation (7) by $\mathbf{E}^{-1}$ generates $\hat{\mathbf{x}} = \mathbf{E}^{-1}\tilde{\mathbf{x}}$, where $\mathbf{E}^{-1}$ is the inverse matrix of $\mathbf{E}$,

$$\mathbf{E}^{-1} = \frac{2}{2N+1} \begin{bmatrix} \frac{1}{2} & \frac{1}{2} & \cdots & \frac{1}{2} \\ \cos(\omega t_1) & \cos(\omega t_2) & \cdots & \cos(\omega t_{2N+1}) \\ \sin(\omega t_1) & \sin(\omega t_2) & \cdots & \sin(\omega t_{2N+1}) \\ \cos(2\omega t_1) & \cos(2\omega t_2) & \cdots & \cos(2\omega t_{2N+1}) \\ \sin(2\omega t_1) & \sin(2\omega t_2) & \cdots & \sin(2\omega t_{2N+1}) \\ \vdots & \vdots & \ddots & \vdots \\ \cos(N\omega t_1) & \cos(N\omega t_2) & \cdots & \cos(N\omega t_{2N+1}) \\ \sin(N\omega t_1) & \sin(N\omega t_2) & \cdots & \sin(N\omega t_{2N+1}) \end{bmatrix}.$$

The treatment of nonlinear term $\hat{f}(\hat{\mathbf{x}})$ in equation (6) is the essence of the HDHB method. In the HDHB method, an approximate relation [4] $\tilde{f}(\tilde{\mathbf{x}}) \approx \mathbf{E}\hat{f}(\hat{\mathbf{x}})$ is introduced to replace the original nonlinear term, where $\tilde{f}(\tilde{\mathbf{x}})$ is the value of $f(\mathbf{x})$ at $2N+1$ discrete time points $\tilde{\mathbf{x}}$. For example, if $f(x) = x^3$ then $\tilde{f}(\tilde{\mathbf{x}})$ is $[x^3(t_1) \ x^3(t_2) \ \cdots \ x^3(t_{2N+1})]^{\mathrm{T}}$. Hence, the algebraic system (6) of the HB method can be simplified as

$$\omega \mathbf{E}\mathbf{A}\mathbf{E}^{-1}\tilde{\mathbf{x}} = \tilde{f}(\tilde{\mathbf{x}}). \tag{8}$$

Equation (8) is the HDHB algebraic system, which consists of $2N + 1$ equations. It is worth noting that obtaining the HDHB algebraic system is much easier (8) than the HB system (6), because all the quantities in the HDHB system are time-domain quantities without having to do complex symbolic calculations required by $\hat{f}(\hat{\mathbf{x}})$.

However, the HDHB method is based on an approximate relation equation $\tilde{f}(\tilde{\mathbf{x}}) \approx \mathbf{E}\hat{f}(\hat{\mathbf{x}})$ which distinguishes it from the HB method. The HDHB method produces unwanted non-physical solutions due to the fact that the high-order harmonic coefficients are aliased into the low-order harmonic components [4,5].

# Part 3. Alternating frequency-time HB

The principle of alternating frequency-time HB (hereafter refered to as AFT) [6] is similar to that of the HDHB method, that is to replace the original HB algebraic system (6) by simple temporal quantities. For processing the linear part, the AFT method is the same as the HDHB method, while for the nonlinear part the AFT and the HDHB methods are different. In contrast to choosing $2N + 1$ collocation nodes in the HDHB method, the sampling rate, viz. number of collocations, in the AFT method is determained by Shannon Sampling Theorem [7]. According to the sampling theorem, when the sampling frequency is greater than twice the maximum frequency of the signal, the Fourier coefficients of the signal can be completely reconstructed without aliasing error. For nonlinear dynamical system with simple polynomial nonlinearities, the sampling rate for the AFT is $2\phi N + 1$, where $\phi$ is the highest nonlinearity and $N$ is the order of truncated Fourier series, viz. the order of the present method. Since $\tilde{\mathbf{x}} = \mathbf{E}\hat{\mathbf{x}}$ always holds, the AFT can reconstruct the harmonic balance form of the entire differential equation $\dot{\mathbf{x}} = f(\mathbf{x}, t)$. However, for system with non-polynomial nonlinearities, how to determine a reasonable sampling rate for the AFT method is still an open problem.

In terms of method implementation, the AFT method is similar to the RHB method. However, the deduction of these two methods is very different. For the nonlinear differential equation $\dot{\mathbf{x}} = f(\mathbf{x}, t)$, the AFT method uses the sampling points of the nonlinear signal to reconstruct the nonlinear term $f(\mathbf{x}, t)$ of the equation. The RHB method gives the minimum collocation number $(\phi + 1)N + 1$ based on rules of aliasing [8,9] and aliasing matrix instead. Using the elements of aliasing matrix, we prove that the RHB method is the optimal reconstruction of the HB method (see Methods). This means that the AFT method is actually oversampled. Therefore, the AFT method is obviously inferior to the RHB method in both simulation calculation and hardware implementation.

# Part 4. Analytical examples for RHB's dealiasing property and a newly discovered conditional identity

It is demonstrated that the aliasing phenomenon of the HDHB method arises from the nonlinear term [4]. In this section, we will show that the RHB method can completely avoid the aliasing error for polynomial nonlinear systems. Take the cubic nonlinear term $x^3$ as an example. For HB1, the approximate solution is $x = x_0 + x_1 \cos(\omega t) + x_2 \sin(\omega t)$. From HB algebraic system (6), the Fourier components $\hat{f}$ for the present case are

$$\hat{f}(\hat{\mathbf{x}}) = \begin{bmatrix} x_0^3 + \frac{3}{2}x_0 x_1^2 + \frac{3}{2}x_0 x_2^2 \\ 3x_0^2 x_1 + \frac{3}{4}x_1^3 + \frac{3}{4}x_1 x_2^2 \\ 3x_0^2 x_2 + \frac{3}{4}x_1^2 x_2 + \frac{3}{4}x_2^3 \end{bmatrix}.$$

For simplicity, the $n$-th order RHB method, i.e., the RHB with $n$ harmonics, is denoted as RHBn. According to the RHB's theory (see Methods), choosing the number of collocations such that $M > (3+1) \times 1 = 4$ can avoid aliasing. For the RHB1, $M = 5$ is sufficient for dealiasing, so that the RHB1's five equidistant collocation nodes are $t_1 = 0$, $t_2 = 2\pi/5$, $t_3 = 4\pi/5$, $t_4 = 6\pi/5$, $t_5 = 8\pi/5$. Correspondingly, the collocation matrix $\mathbf{E}$ and its Moore-Penrose inverse matrix $\mathbf{E}^+$ are

$$\mathbf{E} = \begin{bmatrix} 1 & 1 & 0 \\ 1 & \cos(2\pi/5) & \sin(2\pi/5) \\ 1 & \cos(4\pi/5) & \sin(4\pi/5) \\ 1 & \cos(6\pi/5) & \sin(6\pi/5) \\ 1 & \cos(8\pi/5) & \sin(8\pi/5) \end{bmatrix},$$

and

$$\mathbf{E}^+ = \begin{bmatrix} \frac{1}{5} & \frac{1}{5} & \frac{1}{5} & \frac{1}{5} & \frac{1}{5} \\ \frac{2}{5} & \frac{\sqrt{5}-1}{10} & \frac{-\sqrt{5}-1}{10} & \frac{-\sqrt{5}-1}{10} & \frac{\sqrt{5}-1}{10} \\ 0 & \sqrt{\frac{5+\sqrt{5}}{50}} & \sqrt{\frac{5-\sqrt{5}}{50}} & -\sqrt{\frac{5-\sqrt{5}}{50}} & -\sqrt{\frac{5+\sqrt{5}}{50}} \end{bmatrix}.$$

With the aid of Mathematica, the analytical expression of $\mathbf{E}^+ \tilde{f}(\tilde{\mathbf{x}})$ is

$$\mathbf{E}^+ \tilde{f}(\tilde{\mathbf{x}}) = \begin{bmatrix} x_0^3 + \frac{3}{2}x_0 x_1^2 + \frac{3}{2}x_0 x_2^2 \\ 3x_0^2 x_1 + \frac{3}{4}x_1^3 + \frac{3}{4}x_1 x_2^2 \\ 3x_0^2 x_2 + \frac{3}{4}x_1^2 x_2 + \frac{3}{4}x_2^3 \end{bmatrix} = \hat{f}(\hat{\mathbf{x}}).$$

It is revealed that the Fourier coefficients of the nonlinear term ontained by the RHB1 and the HB1 are exactly the same.

For the HB2, the assumed solution is $x = x_0 + x_1 \cos(\omega t) + x_2 \sin(\omega t) + x_3 \cos(2\omega t) + x_4 \sin(2\omega t)$. The Fourier coefficient vector for the cubic nonlinear term is $\hat{f}(\hat{\mathbf{x}})$. $M$ is selected as 9. It can be proven with the aid of Mathematica that

$$\mathbf{E}^+ \hat{f} = \hat{f},$$

where

$$\hat{f} = \begin{bmatrix} x_0^3 + \frac{3}{4}x_1^2x_3 - \frac{3}{4}x_2^2x_3 + \frac{3}{2}x_1x_2x_4 + \frac{3}{2}x_0x_1^2 + \frac{3}{2}x_0x_2^2 + \frac{3}{2}x_0x_3^2 + \frac{3}{2}x_0x_4^2 \\ 3x_0^2x_1 + 3x_0x_1x_3 + 3x_0x_2x_4 + \frac{3}{4}x_1^3 + \frac{3}{4}x_1x_2^2 + \frac{3}{2}x_1x_3^2 + \frac{3}{2}x_1x_4^2 \\ 3x_0^2x_2 + \frac{3}{4}x_1^2x_2 + \frac{3}{4}x_2^3 - 3x_0x_2x_3 + \frac{3}{2}x_2x_3^2 + 3x_0x_1x_4 + \frac{3}{2}x_2x_4^2 \\ \frac{3}{2}x_0x_1^2 - \frac{3}{2}x_0x_2^2 + 3x_0^2x_3 + \frac{3}{2}x_1^2x_3 + \frac{3}{2}x_2^2x_3 + \frac{3}{4}x_3^3 + \frac{3}{4}x_3x_4^2 \\ 3x_0x_1x_2 + 3x_0^2x_4 + \frac{3}{2}x_1^2x_4 + \frac{3}{2}x_2^2x_4 + \frac{3}{4}x_3^2x_4 + \frac{3}{4}x_4^3 \end{bmatrix}.$$

Through these two examples, it can be found that the RHB method equivalently restructures the HB method by pure collocation scheme.

Finally, take $xy$ as an example. Let the approximate solutions of $x$ and $y$ be $x(t) = x_0 + x_1\cos(\omega t) + x_2\sin(\omega t)$ and $y(t) = y_0 + y_1\cos(\omega t) + y_2\sin(\omega t)$ respectively. Then $xy$ yields

$$xy = \begin{bmatrix} 1 & \cos(\omega t) & \sin(\omega t) & \cos(2\omega t) & \sin(2\omega t) \end{bmatrix} \begin{bmatrix} (2x_0y_0 + x_1y_1 + x_2y_2)/2 \\ x_1y_0 + x_0y_1 \\ x_2y_0 + x_0y_2 \\ (x_1y_1 - x_2y_2)/2 \\ (x_2y_1 + x_1y_2)/2 \end{bmatrix}.$$

Because the $N$-th order HB method retains the first $N$ harmonics, the expression of the Fourier coefficient vector $\hat{f}$ for the HB1 is:

$$\hat{f} = \begin{bmatrix} (2x_0y_0 + x_1y_1 + x_2y_2)/2 \\ x_1y_0 + x_0y_1 \\ x_2y_0 + x_0y_2 \end{bmatrix}.$$

For the present case, the nonlinear term $xy$ is quadratic so that the number of collocations of the RHB method should be $M = 4$. Thus, the expression of time domain $\tilde{f}$ is as follows

$$\tilde{f} = \begin{bmatrix} x(t_1)y(t_1) \\ x(t_2)y(t_2) \\ x(t_3)y(t_3) \\ x(t_4)y(t_4) \end{bmatrix}.$$

Using the same method as the proof of conditional equivalence of the RHB and the HB (see Methods) to deal with $\tilde{f}$, we can get the following results:

$$\tilde{f} = \mathbf{E}\hat{f} + \mathbf{E}_1\hat{f}',$$

where

$$\mathbf{E} = \begin{bmatrix} 1 & \cos t_1 & \sin t_1 \\ 1 & \cos t_2 & \sin t_2 \\ 1 & \cos t_3 & \sin t_3 \\ 1 & \cos t_4 & \sin t_4 \end{bmatrix}, \quad \mathbf{E}_1 = \begin{bmatrix} \cos 2t_1 & \sin 2t_1 \\ \cos 2t_2 & \sin 2t_2 \\ \cos 2t_3 & \sin 2t_3 \\ \cos 2t_4 & \sin 2t_4 \end{bmatrix}, \quad \hat{f}' = \begin{bmatrix} (x_1y_1 - x_2y_2)/2 \\ (x_2y_1 + x_1y_2)/2 \end{bmatrix}.$$

It can be proven that

$$\mathbf{E}^+\tilde{f} = \mathbf{E}^+\mathbf{E}\hat{f} + \mathbf{E}^+\mathbf{E}_1\hat{f}' = \hat{f},$$

where $\mathbf{E}^+\mathbf{E}_1 = \mathbf{E}_A$ is the so-called aliasing matrix. Therefore, we have discovered an important conditional identity $\hat{f} = \mathbf{E}^+\tilde{f}$ (holds for $M > (\phi + 1)N$), based on which the classical HB method is completely reconstructed by the RHB method with simple collocation scheme. It is strictly demonstrated that conditional identity $\hat{f} = \mathbf{E}^+\tilde{f}$ is also applicable to any other polynomial-type nonlinearities.

# Part 5. Accuracy of harmonic balance method for solving Duffing equation

The amplitude-frequency response curve of the Duffing equation is shown in Supp. Fig. 1. When frequency $\omega$ is within the range of approximately $[1.72, 2.44]$, there are three different solutions for each $\omega$. This range is defined as the hysteresis region. Three curves in this region are called upper branch, lower branch, and unstable branch. Numerical integration cannot get the unstable branch. However, semi-analytical methods such as the HB can easily calculate that branch (see Supp. Fig. 1). The results of HB5 and HB7 are basically the same as those of ODE45. This means that the HB method is a powerful and effective method for solving Duffing equations. Take the upper branch of the amplitude-frequency response curve as an example. Curves of error between the results obtained by the numberical integration method and the HB method are shown in Supp. Fig. 2, which shows that the computing results of the HB become more accurate with the increase of the order $N$. It is worth mentioning that the missing portion of the curve of HB7 in the latter part indicates a near-zero error which cannot be plotted in a logarithmic coordinate figure.

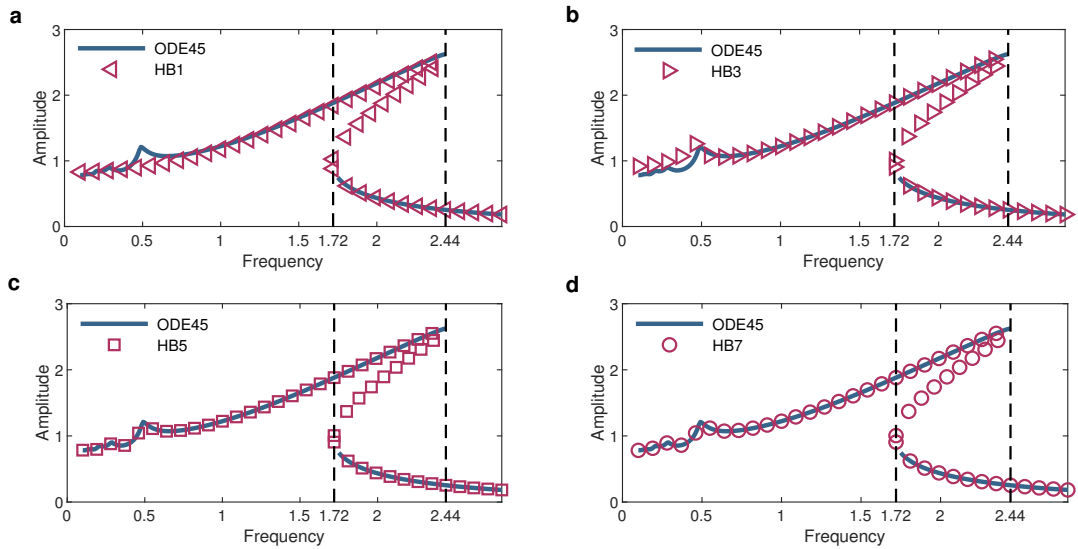

**Supp. Fig. 1.** Amplitude-frequency response curves of Duffing equation calculated by numerical integration method and the HB method. (a) ODE45 and HB1. (b) ODE45 and HB3. (c). ODE45 and HB5. (d). ODE45 and HB7.

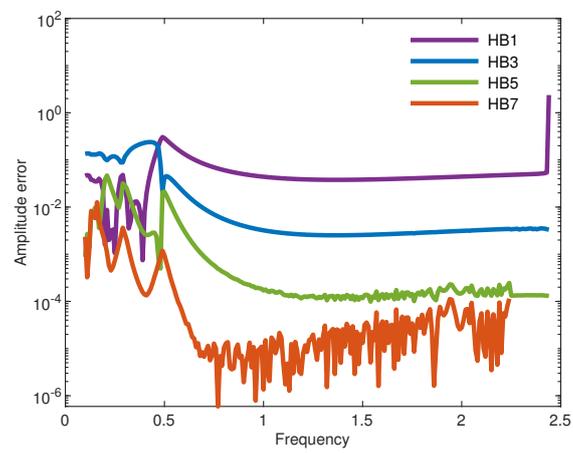

**Supp. Fig. 2.** Errors of amplitude-frequency response curves obtained by the HB method and numerical integration method.

# Part 6. Solving Duffing equation with a mix of nonlinear terms

In addition to the cubic nonlinear Duffing equation, we have verified the effectiveness of the present RHB method on more complex Duffing equations. We set the excitation frequencies between 0.1 and 2.8. $\Delta\omega$ of forward and reverse sweeping is taken as $\Delta\omega = 0.01$. The initial value at the beginning is $X = [x_0 \ x_1 \ \cdots \ x_{2N}]^\mathrm{T} = \mathbf{0}$, and then the result of each step is taken as the initial value of the next step. Supp. Fig. 3a - 3c show the amplitude-frequency response curves of Duffing equation $\ddot{x} + 0.2\dot{x} + x + x^5 = 1.25\sin(\omega t)$ calculated by the HB, the HDHB and the RHB when the order of method is 1, 2 and 3. The curves of the RHB keep an excellent consistency with that of the HB. However, the HDHB is discrepant. In order to evaluate the effect of the RHB on eliminating aliasing, the Monte Carlo simulation is conducted and the associated results are shown in Supp. Fig. 3d - 3e. The first harmonic amplitudes $A_1$ are randomly specified within the range of $\pm 5$. We selected 10000 groups of such random initial values for simulations. The excitation frequency is set to $\omega = 2.0$, at which there are three solutions (two stable solutions and one unstable solution) to the algebraic equation. The RHB avoids the emergence of non-physical solutions and only produces 3 physical solutions. However, the HDHB obtains 25 solutions, 22 of which are non-physical solutions. A similar situation can also be observed in Supp. Fig. 4 for Duffing equation $\ddot{x} + 0.2\dot{x} + x + x^2 + x^3 = 1.25\sin(\omega t)$. A slight difference between Supp. Fig. 3 and Supp. Fig. 4 is that the HB does not converge around $\omega = 1.0$ in Supp. Fig. 3, which may be due to the excessive complexity of the expression of the HB and the accumulation of calculation errors.

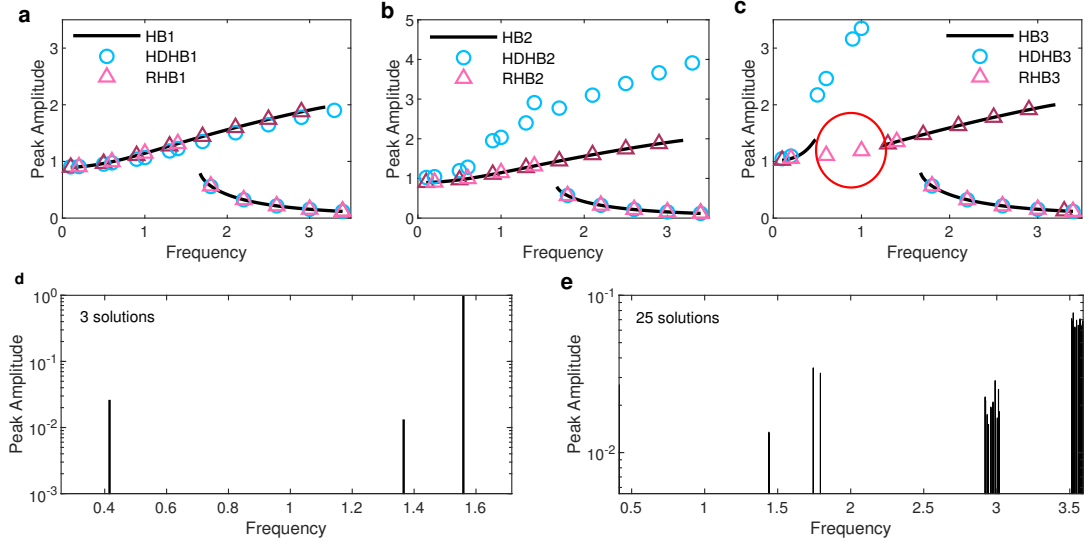

**Supp. Fig. 3.** Comparison of the amplitude-frequency response curves of the quintic Duffing equation $\ddot{x} + 0.2\dot{x} + x + x^5 = 1.25\sin(\omega t)$ calculated by the HDHB, the RHB and the HB methods. (a) First order solutions. (b) Second order solutions. (c) Third order solutions. (d) Monte Carlo histogram for the solutions of the RHB system at $\omega = 2$. (e) Monte Carlo histogram for the solutions of the HDHB system at $\omega = 2$.

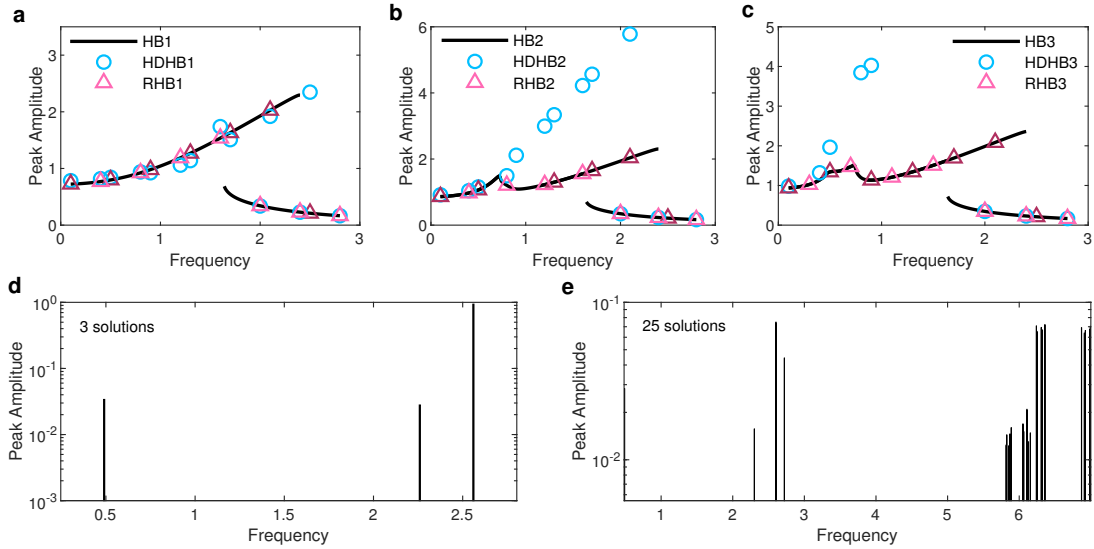

**Supp. Fig. 4.** Comparison of the amplitude-frequency response curves of the quintic Duffing equation $\ddot{x} + 0.2\dot{x} + x + x^2 + x^3 = 1.25\sin(\omega t)$ calculated by the HDHB, the RHB and the HB methods. (a) First order solutions. (b) Second order solutions. (c) Third order solutions. (d) Monte Carlo histogram for the solutions of the RHB system at $\omega = 2$. (e) Monte Carlo histogram for the solutions of the HDHB system at $\omega = 2$.

11

# Part 7. Efficiency of RHB versus HB and AFT

The efficiency of the present RHB method is compared with the HB and the AFT methods. The HB method involves tedious symbolic calculations for nonlinear terms, which causes expensive computational cost. To make the most the HB method, we use Mathematica for symbolic calculations and MATLAB for numerical calculations. Supplementary Table 1 and 2 provide the computing time of solving the cubic and quintic Duffing equations $\ddot{x} + 0.2\dot{x} + x + x^3 = 1.25\sin(\omega t)$ and $\ddot{x} + 0.2\dot{x} + x + x^5 = 1.25\sin(\omega t)$, respectively. The computing time of the HB is always longer than that of the RHB. As the order of method increases, the computing time required by the HB increases almost exponentially. However, the additional computational cost required by the RHB is insignificant. For the six-order computations, the computing time ratio between the HB and the RHB blows up to 1460.

**Supplementary Table 1:** Computing time of cubic nonlinear Duffing equation $\ddot{x} + 0.2\dot{x} + x + x^3 = 1.25\sin(\omega t)$ analysis using the HB and the RHB.

| Methods | Computing time (s) (MATLAB and Mathematica) | Ratio (HB/RHB) |
|---|---|---|
| HB1 | 0.006 + 0.016 | 2.200 |
| RHB1 | 0.010 | |
| HB2 | 0.011 + 0.031 | 4.200 |
| RHB2 | 0.010 | |
| HB3 | 0.017 + 0.047 | 6.400 |
| RHB3 | 0.010 | |
| HB4 | 0.036 + 0.062 | 8.909 |
| RHB4 | 0.011 | |
| HB5 | 0.060 + 0.109 | 13.000 |
| RHB5 | 0.013 | |
| HB6 | 0.102 + 0.190 | 20.857 |
| RHB6 | 0.014 | |

Supplementary Table 3 gives the computing time of solving the cubic Duffing equation $\ddot{x} + 0.2\dot{x} + x + x^3 = 1.25\sin(\omega t)$ using the AFT and the RHB. It can be seen that the AFT always costs longer computing time. As the order of method increases, the ratio of the computing time between the AFT to the RHB increases. When high order approximation is required (greater than 100) the computational efficiency of the RHB is about twice that of the AFT.

**Supplementary Table 2:** Computing time of quintic nonlinear Duffing equation $\ddot{x} + 0.2\dot{x} + x + x^5 = 1.25\sin(\omega t)$ analysis using the HB and the RHB.

| Methods | Computing time (s) (MATLAB and Mathematica) | Ratio (HB/RHB) |
|---|---|---|
| HB1  | 0.009 + 0.016 | 2.273 |
| RHB1 | 0.011 | |
| HB2  | 0.026 + 0.047 | 6.636 |
| RHB2 | 0.011 | |
| HB3  | 0.143 + 0.266 | 34.083 |
| RHB3 | 0.012 | |
| HB4  | 0.541 + 0.984 | 117.308 |
| RHB4 | 0.013 | |
| HB5  | 1.814 + 3.687 | 423.154 |
| RHB5 | 0.013 | |
| HB6  | 12.296 + 8.155 | 1460.786 |
| RHB6 | 0.014 | |

**Supplementary Table 3:** Computing time of cubic nonlinear Duffing equation $\ddot{x} + 0.2\dot{x} + x + x^3 = 1.25\sin(\omega t)$ analysis using the AFT and the RHB.

| Methods | Computing time (s) (MATLAB) | Ratio (AFT/RHB) |
|---|:---:|:---:|
| AFT10 | 0.014 | |
| RHB10 | 0.013 | 1.077 |
| AFT30 | 0.034 | |
| RHB30 | 0.026 | 1.301 |
| AFT50 | 0.067 | |
| RHB50 | 0.049 | 1.367 |
| AFT70 | 0.124 | |
| RHB70 | 0.076 | 1.632 |
| AFT90 | 0.216 | |
| RHB90 | 0.121 | 1.785 |
| AFT110 | 0.403 | |
| RHB110 | 0.192 | 2.099 |
| AFT130 | 0.864 | |
| RHB130 | 0.362 | 2.387 |
| AFT150 | 1.554 | |
| RHB150 | 0.610 | 2.548 |
| AFT170 | 2.595 | |
| RHB170 | 0.998 | 2.600 |
| AFT190 | 3.742 | |
| RHB190 | 1.430 | 2.617 |

# Part 8. Polynomialization of Rayleigh-Plesset equation

The Rayleigh-Plesset (R-P) equation [10, 11] used in this study is

$$R\ddot{R} = -\frac{3}{2}\dot{R}^2 - \frac{4\eta_L \dot{R}}{\rho_L R} - \frac{2\sigma}{\rho_L R} + \frac{p_{\text{air}} R_0^3}{\rho_L R^3} + \frac{p_{\text{sat}} - p_{\text{far}}}{\rho_L} - \frac{p_{\text{amp}}}{\rho_L}\cos(\omega t), \qquad (9)$$

where $R$ is the radius of cavity; $\eta_L$ is the dynamic viscosity of water; $\rho_L$ is the density of water; $\sigma$ is the interfacial tension between the liquid and vapor phases of pure water; $p_{\text{air}}$ is the initial pressure of dry air; $R_0$ is the initial radius of cavity; $p_{\text{sat}}$ is the vapor pressure of water; $p_{\text{far}}$ is the pressure at farfield; $p_{\text{amp}}$ is the standard atmospheric pressure. Supplementary Table 4 shows the parameters of equation (9).

**Supplementary Table 4:** Parameters used in Rayleigh-Plesset equation case.

| Parameter | Value |
| --- | --- |
| $\eta_L$ | $1.002 \times 10^{-3}$ Pa·s |
| $\rho_L$ | 998.2 $kg/m^3$ |
| $\sigma$ | $7.28 \times 10^{-2}$ N/m |
| $p_{\text{air}}$ | $2.3267 \times 10^5$ Pa |
| $p_{\text{sat}}$ | 2338.8 Pa |
| $p_{\text{far}}$ | $1.01325 \times 10^5$ Pa |
| $p_{\text{amp}}$ | $1.01325 \times 10^5$ Pa |
| $\omega$ | $1.0472 \times 10^5$ Hz |
| $R_0$ | $4.50 \times 10^{-6}$ m |

For simplicity, the following parameters are introduced as:

$$A = \frac{4\eta_L}{\rho_L}, \ B = \frac{2\sigma}{\rho_L}, \ C = \frac{p_{\text{air}} R_0^3}{\rho_L}, \ D = \frac{p_{\text{sat}} - p_{\text{far}}}{\rho_L}, \ E = \frac{p_{\text{amp}}}{\rho_L}.$$

Thus, the R-P equation (9) turns out to be

$$R\ddot{R} = -\frac{3}{2}\dot{R}^2 - A\frac{\dot{R}}{R} - B\frac{1}{R} + C\frac{1}{R^3} + D - E\cos(\omega t). \qquad (10)$$

Obviously, the R-P equation is not a kind of polynomial nonlinear system; therefore the classical HB method cannot solve this equation. Although the HDHB and the AFT methods can handle this problem but they will generate non-ignorable aliasing errors due to the nonlinearity-induced infinite harmonics. To this end, the R-P equation (10) is transformed into a corresponding polynomial form by a recasting technique [12, 13], and then solved by the RHB.

Upon defining

$$u = R, \ v = \dot{R}, \ x = \frac{1}{R},$$

equation (10) is immediately transformed into polynomial form:

$$\begin{cases} \dot{u} = v, \\ \dot{v} = -\dfrac{3}{2}v^2 x - Avx^2 - Bx^2 + Cx^4 + Dx - Ex\cos(\omega t), \\ 0 = ux - 1. \end{cases} \quad (11)$$

Specifically, equation (10) has been euivalently converted into a fourth-order polynomial system (11) which can be readily solved by the RHB method without suffering from any aliasing error. Therefore, the RHB method in conjunction with the recasting technique is a highly accurate and efficient tool for solving the periodic solutions of non-polynomial nonlinear systems. Supp. Fig. 5 shows the comparison of the residual error curves of solving equations (10) and (11) with different NAE solvers (Newton-Raphson method and global optimal iterative algorithm (GOIA) [14]). It can be seen that the conversion greatly improves the calculation accuracy. The original equation does not even converge when being solved by the Newton method. Although convergent by the Jacobian-inverse-free GOIA solver, the result accuracy of solving the original equation is two orders of magnitude lower than that of solving the recast equation.

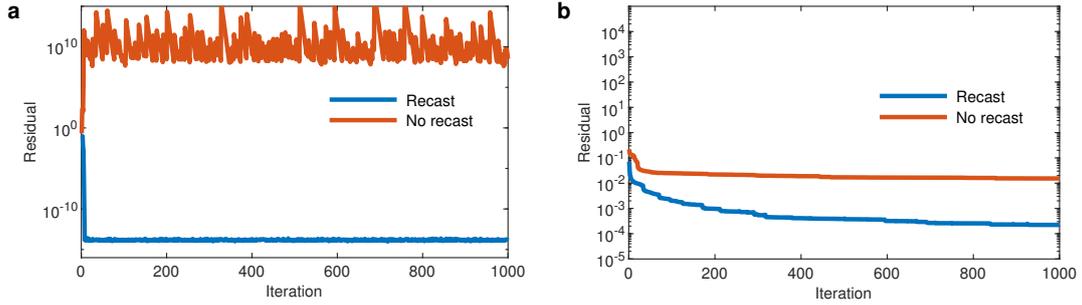

**Supp. Fig. 5.** Residual error versus iterations for solving the original and recast systems by different NAE solvers with the wavenumber equals 150. (a) NRM. (b) GOIA.

# Part 9. Effect of preconditioning of the R-P equation

Owing to the high efficiency, extremely high order approximations can be achieved in the RHB method. It is found that the R-P dynamical response contains very high frequency components, and even the 200-th harmonic part plays a non-ignorable role (see Fig. 4). Due to the coexistence of a wide spectrum of frequencies, the resulting RHB algebraic system becomes ill-conditioned which severely damages the computational accuracy. In addition to the wide frequency range, it is noticed that the radius of the cavity $R$, its reciprocal $x = 1/R$, and the velocity of the cavity span more than 10 orders of magnitudes, which also amplifies the system's ill-conditionedness. Therefore, preconditioning is necessary for the RHB method. Specifically, we introduce

$$\bar{u} = \frac{u}{R_0}, \ \bar{v} = v, \ \bar{x} = xR_0, \ \bar{\omega} = \omega R_0, \ \bar{t} = \frac{t}{R_0}.$$

Thus, equation (11) becomes

$$\begin{cases} \mathrm{d}\bar{u}/\mathrm{d}\bar{t} = \bar{v}, \\ \mathrm{d}\bar{v}/\mathrm{d}\bar{t} = -\frac{3}{2}\bar{v}^2\bar{x} - \frac{A}{R_0}\bar{v}\bar{x}^2 - \frac{B}{R_0}\bar{x}^2 + \frac{C}{R_0^3}\bar{x}^4 + D\bar{x} - E\bar{x}\cos(\bar{\omega}\bar{t}), \\ 0 = \bar{u}\bar{x} - 1, \end{cases} \quad (12)$$

which is readily solved by the RHB method. Supp. Fig. 6 shows the effect of preconditioning on computing accuracy. We choose $N = 150$ for computation. It can be seen that the accuracy of the computation results is very low without preconditioning. Preconditioning can improve the computing accuracy by about 2 orders of magnitude.

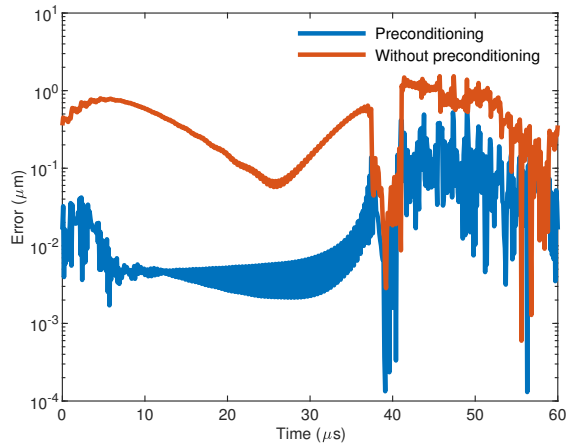

**Supp. Fig. 6.** Error curves of the solutions obtained with and without preconditioning against the benchmark numerical integration solution.

# Part 10. Polynomialization of circular restricted three body problem system

The governing equation of the circular restricted three body problem (CRTBP) [15] is

$$\begin{cases} \ddot{x} - 2\dot{y} = -\dfrac{1}{\gamma^2}\bar{U}_x, \\ \ddot{y} + 2\dot{x} = \dfrac{1}{\gamma^2}\bar{U}_y, \\ \ddot{z} = \dfrac{1}{\gamma^2}\bar{U}_z, \end{cases} \quad (13)$$

where

$$\bar{U} = -\frac{1}{2}[(\gamma x + X)^2 + (\gamma y)^2] - \frac{1-\mu}{r_1} - \frac{\mu}{r_2} - \frac{1}{2}\mu(1-\mu),$$

$$r_1 = \sqrt{(\gamma x + X + \mu)^2 + (\gamma y)^2 + (\gamma z)^2},$$

$$r_2 = \sqrt{(\gamma x + X - 1 + \mu)^2 + (\gamma y)^2 + (\gamma z)^2},$$

and $\bar{U}_x$, $\bar{U}_y$ and $\bar{U}_z$ are the partial derivatives of $\bar{U}$ in $x$, $y$ and $z$ directions respectively, $X$ is the coordinate of the current libration point, $\gamma$ is the dimensionless distance between the current libration point and the nearest central celestial body. Apparently, the CRTBP is a non-polynomial type nonlinear system which cannot be efficiently solved by existing methods. Herein, we use the presently proposed recasting technique plus the RHB method to attack the CRTBP problem.

For recasting purpose, we need to introduce new variables

$$u = r_1^3, \ v = r_2^3.$$

Thus, equation (13) can be equivalently converted into the following polynomial type differential-algebraic system:

$$\begin{cases} \ddot{x} - 2\dot{y} = -\dfrac{1}{\gamma}[-(\gamma x + X_i) + (1-\mu)(\gamma x + X_i + \mu)u + \mu(\gamma x + X_i - 1 + \mu)v], \\ \ddot{y} + 2\dot{x} = y - (1-\mu)yu - \mu yv, \\ \ddot{z} = -(1-\mu)zu - \mu zv, \\ 0 = ur_1^3 - 1, \\ 0 = vr_2^3 - 1, \\ 0 = (\gamma x + X_i + \mu)^2 + (\gamma y)^2 + (\gamma z)^2 - r_1^3, \\ 0 = (\gamma x + X_i - 1 + \mu)^2 + (\gamma y)^2 + (\gamma z)^2 - r_2^3. \end{cases} \quad (14)$$

This transformation supplies the RHB method with polynomial nonlinear equations which successfully remedies the aliasing error arising from solving non-polynomial systems. In computations, the state of the calculation results at a certain time is used to produce the initial value. Supp. Fig. 7a-7d show the error of the RHB results against a high-fidelity benchmark result (RHB50). It can be seen that under the same convergence criterion, the computational accuracy of the RHB with recasting technique is several orders of magnitude better than that without recasting.

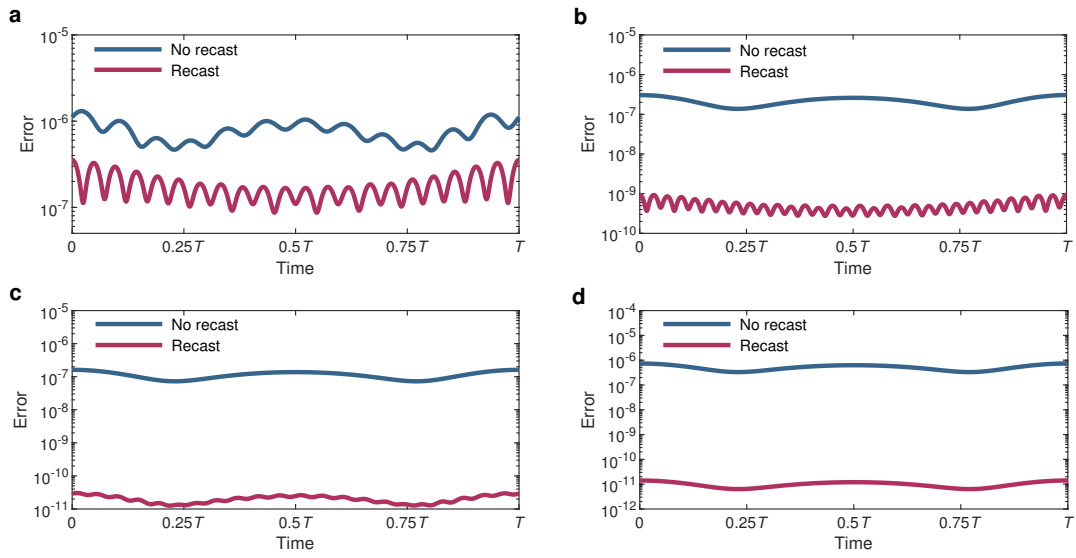

**Supp. Fig. 7.** Computational errors for solving original and recast CRTBP systems with the different RHB methods. (a) RHB10. (b) RHB15. (c) RHB20. (d) RHB25.

# Part 11. Differential correction method for finding halo orbit

Currently, the most commonly used method for preliminary design of the halo orbit is the differential correction (DC) method [16]. The main idea of the DC method is to continuously adjust the state at the initial moment until the final orbit is almost periodic by satisfying some prescribed periodicity criterion. Hence, a reasonable initial guess is very important for the DC method. It is credited to Ricardson [17] that a relatively good initial guess is obtained by solving an approximate CRTBP system, where the nonlinear term of the original CRTBP model is expanded by Legendre series and the higher order terms (greater than fourth) are ignored to generate the third-order approximate CRTBP model as follows

$$\begin{cases} \ddot{x} - 2\dot{y} - (1+2c_2)x = \frac{3}{2}c_3(2x^2 - y^2 - z^2) + 2c_4 x(2x^2 - 3y^2 - 3z^2), \\ \ddot{y} + 2\dot{x} + (c_2 - 1)y = -3c_3 xy - \frac{3}{2}c_4 y(4x^2 - y^2 - z^2), \\ \ddot{z} + c_2 z = -3c_3 xz - \frac{3}{2}c_4 z(4x^2 - y^2 - z^2), \end{cases} \quad (15)$$

where the coefficients $c_n(\mu)$ are

$$c_n = \frac{1}{\gamma^3}[(-1)^n \mu + (-1)^n \frac{(1-\mu)\gamma^{n+1}}{(1+\gamma)^{n+1}}].$$

We use the DC method to correct the initial value of the Richardson's third-order approximation solution of the halo orbit. Tranforming $x = x$, $y = -y$, $z = z$, $\dot{x} = -\dot{x}$, $\dot{y} = \dot{y}$, $\dot{z} = \dot{z}$ does not change the equation (13). Therefore, the halo orbit is symmetric about the $xz$ plane. Due to symmetry, halo orbits have a velocity perpendicular to the $xz$ plane at the point of intersection with the plane (hereafter referred to as the crossing point), that is $\dot{x} = \dot{z} = 0$. There are two crossing points. According to the above content, the coordinates can be set as $\bar{\mathbf{x}}_0 = [x_0 \; 0 \; z_0 \; 0 \; \dot{y}_0 \; 0]^\mathrm{T}$. Assuming that this value is the exact initial value of halo orbit, the state when the dynamic equation integral $T/2$ ($T$ is the period of halo orbit) reaches another crossing point can be recorded as

$$\bar{\mathbf{x}}_d(\frac{T}{2}) = [x_0 \; 0 \; z_0 \; 0 \; \dot{y}_0 \; 0]^\mathrm{T},$$

which is the target state. However, $\bar{\mathbf{x}}_0$ is the approximate initial value. The true orbit cannot guarantee the vertical crossing of the $xz$ plane, and the time for the first crossing of the $xz$ plane is not necessarily equal to $T/2$. At this time, the state at another crossing point is set to $\bar{\mathbf{x}}_1(t_1) = [x_1 \; 0 \; z_1 \; \dot{x}_1 \; \dot{y}_1 \; \dot{z}_1]^\mathrm{T}$. Let

$$\delta \bar{\mathbf{x}}_1 = \bar{\mathbf{x}}_d - \bar{\mathbf{x}}_1(t_1) = [\delta x_1 \; 0 \; \delta z_1 \; \delta \dot{x}_1 \; \delta \dot{y}_1 \; \delta \dot{z}_1]^\mathrm{T},$$

where $\delta \dot{x}_1 = -\dot{x}_1$ and $\delta \dot{z}_1 = -\dot{z}_1$. The DC method needs to solve the following equation

$$\begin{bmatrix} \delta x_0 \\ \delta \dot{y}_0 \\ \delta t_1 \end{bmatrix} = \begin{bmatrix} \Phi_{21} & \Phi_{25} & \dot{y}_1 \\ \Phi_{41} & \Phi_{45} & \ddot{x}_1 \\ \Phi_{61} & \Phi_{65} & \ddot{z}_1 \end{bmatrix}^{-1} \begin{bmatrix} 0 \\ \delta \dot{x}_1 \\ \delta \dot{z}_1 \end{bmatrix},$$

where $\delta x_0$, $\delta \dot{y}_0$, $\delta t_1$ are the correction variables of $x_0$, $\dot{y}_0$ and $t_1$, respectively; $\Phi_{ij}$ is the corresponding element in the state transition matrix $\mathbf{\Phi}(t;t_0)$. The state transition matrix is obtained by numerical integration of the following equation.

$$\frac{\mathrm{d}\mathbf{\Phi}(t;t_0)}{\mathrm{d}t} = \begin{bmatrix} 0 & 0 & 0 & 1 & 0 & 0 \\ 0 & 0 & 0 & 0 & 1 & 0 \\ 0 & 0 & 0 & 0 & 0 & 1 \\ -\bar{U}_{xx} & -\bar{U}_{xy} & -\bar{U}_{xz} & 0 & 2 & 0 \\ -\bar{U}_{xy} & -\bar{U}_{yy} & -\bar{U}_{yz} & -2 & 0 & 0 \\ -\bar{U}_{xz} & -\bar{U}_{yz} & -\bar{U}_{zz} & 0 & 0 & 0 \end{bmatrix} \mathbf{\Phi}(t;t_0),$$

$$\mathbf{\Phi}(t_0;t_0) = \mathbf{I}_6.$$

Adding the correction calculated above to $\bar{\mathbf{x}}_0$, and repeating this correction process until the convergence criterion satisfied, a nearly periodic halo orbit can be obtained. However, the orbit by the DC method is essentially not closed as opposed to the RHB method. Supp. Fig. 8. gives the halo orbit calculated by the DC method and corresponding Richardson solution.

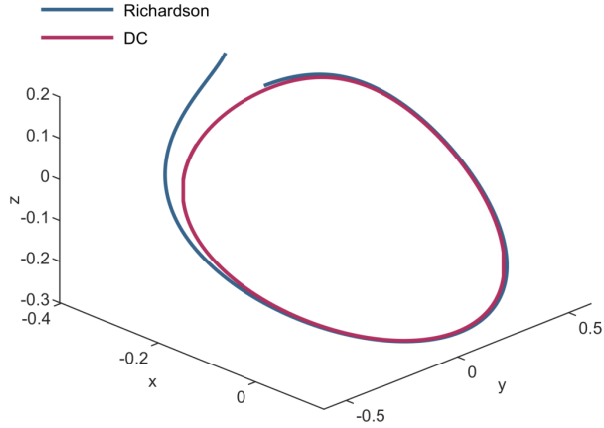

**Supp. Fig. 8.** Richardson approximate solution and orbits corrected by the DC Method.

# Part 12. Comparation of CRTBP halo orbit by RHB and China's Queqiao orbit

Supplementary Table 5 shows 9 real flight data of the Queqiao satellite (provided by the authors' collaborator of Beijing Aerospace Control Center). For the purpose of comparison, the halo orbit (with $A_z$ equal to 13000km) obtained by the RHB method is converted to the J2000 coordinate system, as shown in Supp. Fig. 9, where the position of the moon is determined by the JPL ephemeris. It can be seen from Supp. Fig. 9 that the CRTBP halo orbit by the RHB method agrees with the real flight data. Concretely speaking, the discrepancy of the result by the present method is less than 8% compared with the real fight orbit. It is worth noting the Queqiao orbit is not designed as an exactly periodic orbit but the periodic orbit is necessary to offer the mission with a nominal orbit. A highly periodic orbit can make the subsequent control process fuel-saving. In this sense, the highly periodic CRTBP halo orbit by the RHB method is very useful in the preliminary orbit design phase.

**Supplementary Table 5:** The position of Queqiao satellite in J2000 coordinate system.

| Epoch | $x$/m | $y$/m | $z$/m |
|---|---|---|---|
| 2018-07-23 22:10:00.000 | -92138192.4369 | -425522124.8242 | -153311530.7649 |
| 2018-07-25 09:00:00.000 | 42614066.5343 | -436874344.1939 | -175929648.0113 |
| 2018-07-26 09:00:00.000 | 131322592.6252 | -423162110.0119 | -181830273.0097 |
| 2018-07-27 09:00:00.000 | 213566989.3990 | -393613469.0381 | -179827294.7594 |
| 2018-07-28 09:00:00.000 | 287066726.3491 | -349639391.9713 | -170074370.3026 |
| 2018-07-31 01:00:07.125 | 424439656.4903 | -172689511.5179 | -108928192.8875 |
| 2018-08-02 09:00:00.000 | 448818102.1210 | 32842816.9061 | -21961372.0982 |
| 2018-08-06 09:00:00.000 | 197098051.4794 | 353371139.2604 | 123586797.3687 |
| 2018-08-07 08:00:00.000 | 94139123.5574 | 385123607.3258 | 138247395.6059 |

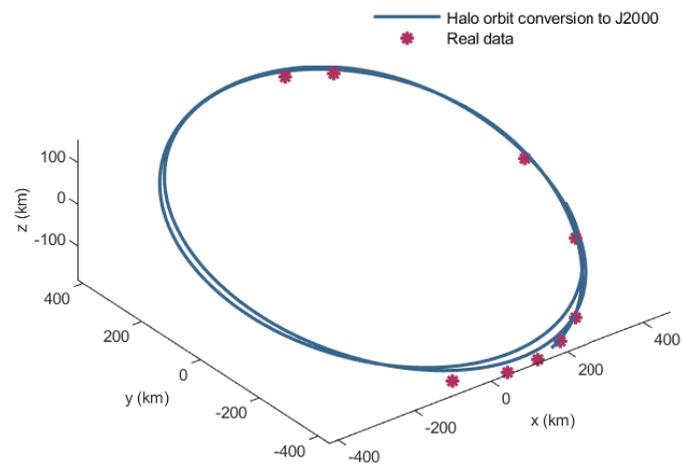

**Supp. Fig. 9.** Orbits comparison in J2000 coordinate system.

# Part 13. Orbit family obtained by the RHB method with parameter sweeping

The RHB method generates additional amplitude-frequency response branches (see Supp. Fig. 10a-10b) in addition to the four families in Fig. 6. Their corresponding orbital families are shown in Supp. Fig. 10c-10f, where Supp. Fig. 10e-10f are the partial enlargements of the red rectangles in Supp. Fig. 10c-10d, respectively.

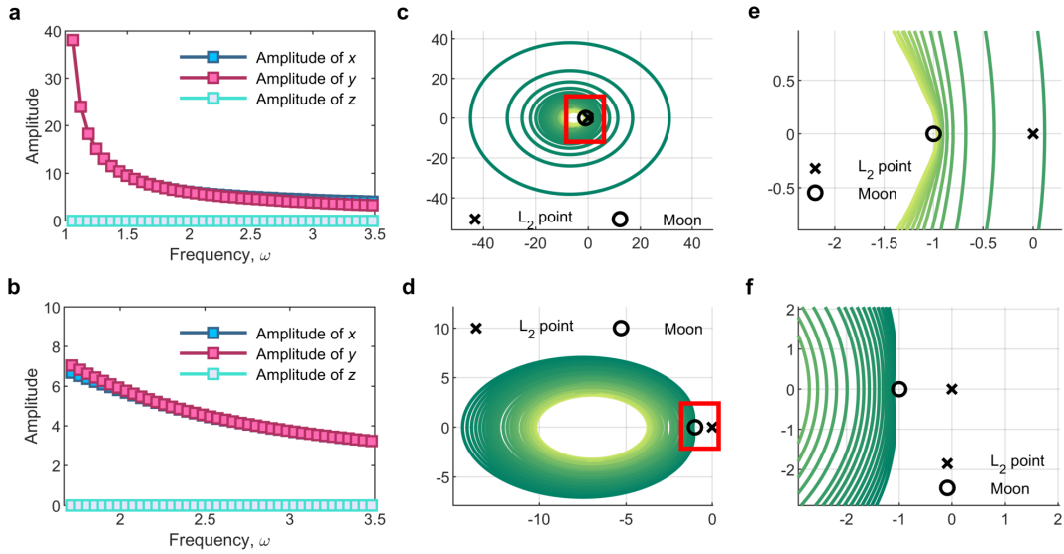

**Supp. Fig. 10.** Amplitude-frequency response curves and orbits for additional orbit family. (a), (b) Amplitude-frequency response curves. (c), (d) Orbit families. (e), (f) The partial enlargements of (c) and (d).



\end{document}